\documentclass[12pt]{amsart}
\usepackage{amscd,amssymb}
\usepackage{graphicx,enumerate}

\usepackage[bbgreekl]{mathbbol}

\usepackage{color}
\usepackage[arrow,matrix,arrow,ps,color,line,curve,frame]{xy}

\usepackage{pgf, tikz}
\usepackage{url}
\usepackage{enumerate}

\usepackage[colorinlistoftodos,bordercolor=red,backgroundcolor=red!20,linecolor=red,textsize=scriptsize]{todonotes}

\let\oldlabel=\label

\def\prellabel{\marginparsep=1em
    \def\label##1{\oldlabel{##1}\ifmmode\else\ifinner\else
         \marginpar{{\footnotesize\ \\ \tt
                    ##1}}\fi\fi}}
                    
\def\AA{\mathbb{A}}

\def\pr{\operatorname{pr}}

\def\B{\operatorname{B}}

\def\SF{\operatorname{SF}}

\def\conv{\operatorname{conv}}

\def\ee{\mathbb{e}}

\def\Im{\operatorname{Im}}
\def\vertex{\operatorname{vert}}
\def\inte{\operatorname{int}}
\def\Hilb{\operatorname{Hilb}}
\def\Fano{\operatorname{Fano}}
\def\Aff{\operatorname{Aff}}

\def\BB{\mathbb{B}}

\def\rank{\operatorname{rank}}

\def\spec{\operatorname{Spec}}
\def\proj{\operatorname{Proj}}

\def\Aut{\operatorname{Aut}}

\def\kk{\mathbf{k}}

\def\MM{\mathbf{M}}

\def\PPi{\mathbb\Pi}

\def\gp{\operatorname{gp}}

\def\l{\operatorname{\bf l}}

\def\gr{\operatorname{gr}}

\def\RR{\mathbb R}

\def\PP{\mathbb P}

\def\ZZ{\mathbb Z}
\def\NN{\mathbb N}

\def\MM{\mathbb M}

\let\phi=\varphi
\let\epsilon=\varepsilon

\advance\headheight1.15pt \marginparwidth=2.5cm

\newtheorem{lemma}{Lemma}[section]
\newtheorem{corollary}[lemma]{Corollary} 
\newtheorem{theorem}[lemma]{Theorem}

\theoremstyle{definition} 
\newtheorem{definition}[lemma]{Definition} 
\newtheorem{remark}[lemma]{Remark}
\newtheorem{example}[lemma]{Example} 
\newtheorem{question}[lemma]{Question}

\begin{document}

\title[Bottom complexes]{Bottom complexes}

\author{Joseph Gubeladze}

\address{Department of Mathematics\\
         San Francisco State University\\
         1600 Holloway Ave.\\
         San Francisco, CA 94132, USA}
\email{soso@sfsu.edu}




\begin{abstract}
The \emph{bottom complex} of a finite polyhedal pointed rational cone is the lattice polytopal complex of the compact faces of the convex hull of nonzero lattice points in the cone. The algebra, associated to the bottom complex of a cone, defines a flat deformation of the affine toric variety, associated to the polyhedral cone, set-theoretically. We describe three explicit infinite families of abstract polytopal complexes, defining such flat deformations scheme-theoretically.
\end{abstract}

\maketitle

\section{Introduction} A finite, pointed, rational cone $C\subset\RR^d$, i.e., a cone generated by a finite subset of $\ZZ^d$ and containing no non-zero linear subspace, gives rise to a \emph{lattice polyhedal complex:} the complex of the compact faces of the convex hull of $(C\cap\ZZ^d)\setminus\{0\}$. This is the \emph{bottom complex} of $C$, denoted by $\BB(C)$. Now assume we are given a lattice polytopal complex $\PPi$ as a set of lattice polytopes along with data on how polytopes are glued along common faces; see Section \ref{Complexes} for the formal definitions. When does there exist a cone $C$, whose bottom complex is isomorphic to $\PPi$? This is a nontrivial question, even for simplicial complexes equipped with the coarsest lattice structure.

There is an alternative formulation of this question in terms of the \emph{polyhedral algebra} $\kk[\BB(C)]$ of the bottom complex $\BB(C)$, which extends the notion of Stanley-Reisner ring $\kk[\Delta]$ of a simplicial complex $\Delta$ (Section \ref{Complexes}). For a cone $C\subset\RR^d$ and an algebraically closed field $\kk$, the Rees algebra with respect to the maximal monomial ideal $I\subset\kk[C\cap\ZZ^d]$, defines a flat deformation of the toric ring $\kk[C\cap\ZZ^d]$ to $\gr_I(\kk[C\cap\ZZ^d])$. On the one hand, the polyhedral algebra of a lattice polytopal complex determines uniquely the underlying complex (Theorem \ref{isomorphism}) and, on the other hand, $\spec(\gr_I(\kk[C\cap\ZZ^d]))$ and $\spec(\kk[\BB(C)])$ agree set-theoretically (Theorem \ref{Folklore}). Thus our problem asks to characterize lattice polytopal complexes, defining Rees deformations of affine toric varieties set-theoretically. But one can go one step further and ask for a characterization of the lattice polytopal complexes, defining such deformations \emph{scheme-theoretically}. We call such complexes \emph{reduced bottom}.

A necessary condition for a lattice polytopal complex to be reduced bottom is that the underlying topological space must be a topological ball. In dimension one this is also sufficient, an old observation in toric geometry \cite[Section 1.6]{Oda}; see also Remark 
\ref{1-dimensional}. In high dimensions the topological condition falls far short of being sufficient, even for simplicial complexes (Section \ref{Obstructions}).

In this paper we give several explicit infinite families of reduced bottom polytopal complexes. This includes:
\begin{enumerate}[\rm$\bullet$]
\item
Complexes of arbitrary dimension, defined by a shellability like condition (Theorem \ref{MAIN}), which include as a proper subclass the \emph{stacked} complexes, built up inductively by stacking one polytope at a time along a common face;
\item An exhaustive characterization of the reduced bottom realizations, satisfying a natural convexity condition, of the pyramid over a general simplicial sphere (Theorem \ref{Fano_bottom}); such realizations only exist if the simplicial sphere is combintorially equialent to the  boundary complex of a smooth Fano polytope;
\item A series of reduced bottom simplicial discs (Theorem \ref{Delta}), not covered by the theorems above; we do not know whether all simplicial discs are reduced bottom, but the positive answer can be given to the weakar version, when the normality condition is relaxed for the underlying monoid (Theorem \ref{evidence}). 
\end{enumerate}

\section{Lattice polytopal complexes and their algebras}\label{Complexes}

\subsection{Cones and polytopes}

We use the following notation and terminology:

\begin{enumerate}[\rm$\centerdot$]
\item
$\NN=\{1,2,\ldots\}$, $\ZZ_{\ge n}=\{n,n+1,n+2,\ldots\}$, $\ZZ_+=\ZZ_{\ge0}$, $\RR_+$ is the set the non-negative reals, and $\RR_{>0}$ that of the positive ones.
\item For a subset $X$ of a finite dimensional real space, its affine and convex hulls are denoted by $\Aff(X)$ and $\conv(X)$, respectively; the \emph{conical set} over $X$ is $\RR_+X=\{ax\ |\ a\in\RR_+,\ x\in X\}$; if $X$ is convex its relative interior is denoted by $\inte(X)$.
\item $\ee_i$ is the $i$-th standard basic vector in $\RR^d$.
\end{enumerate}

\medskip\noindent\emph{Affine spaces and polytopes:}

\begin{enumerate}[\rm$\centerdot$]
\item An \emph{affine space} is a shifted linear subspace, i.e., a subset of a real vector space $V$ of the form $H=x+U$, where $x\in V$ and $U\subset V$ is a subspace; an \emph{affine map} is a shifted linear map.
\item An \emph{affine lattice} $\Lambda$ in an affine space $H=x+U$ (notation as above) is a subset of the form $x+\Lambda_0$, where $\Lambda_0\subset U$ is a lattice;
\item All our polytopes are convex in their ambient vector spaces; for a polytope $P$, its vertex set is denoted by $\vertex(P)$;
\item For an affine lattice $\Lambda$, a polytope $P$ is called a \emph{$\Lambda$-polytope} if $\vertex(P)\subset\Lambda$; a \emph{lattice polytope} refers to a $\ZZ^d$-polytope for some $d\in\NN$.
\item Two lattice polytopes $P\subset\RR^d$ and $P\rq{}\subset\RR^{d\rq{}}$ are \emph{unimodularily equivalent} if there is an affine isomorphism $P\to P\rq{}$, yielding an affine isomorphism $\Aff(P)\cap\ZZ^d\to\Aff(P\rq{})\cap\ZZ^{d\rq{}}$.
\item For a lattice $\Lambda$ in a vector space $V$, we say that an affine hyperplane $H\subset V$ is on \emph{lattice distance one} from a point $x\in\Lambda$ if there are no elements of $\Lambda$ strictly between $H$ and its parallel translate through $x$. 
\end{enumerate}

\medskip\noindent\emph{Normal polytopes:}
\begin{enumerate}[\rm$\centerdot$]
\item
A lattice polytope $P\subset\RR^d$ is \emph{normal} if the following implication holds:
\begin{align*}
c\in\NN,\ z\in(cP)\cap\ZZ^d\ \Longrightarrow\ 
\exists y_1,\ldots,y_c\in P\cap\ZZ^d,\quad y_1+\cdots+y_c=z;
\end{align*}
\item More generally, for an affine space $H$ and an affine lattice $\Lambda\subset H$, a $\Lambda$-polytope $P$ is called \emph{$\Lambda$-normal} if $P$ becomes normal upon making some (equivalently, any) point of $\Lambda$ into the origin.
\end{enumerate}

Normal polytopes are central objects of study in toric algebraic geometry and combinatorial commutative algebra \cite{Kripo}. 

The simplest normal polytopes are \emph{unimodular} simplices, i.e, simplices of the form $\conv(x_1,\ldots,x_k)$, where $\{x_j-x_i\}_{j\not=i}$ is a part of a basis of the lattice of reference for some (equivalently, any) $1\le i\le k$.

\medskip\noindent\emph{Cones:}
\begin{enumerate}[\rm$\centerdot$]
\item A \emph{cone} is an $\RR_+$-submodule $C\subset\RR^d$; our cones are always assumed to be \emph{finite, rational,} and \emph{pointed}, i.e., of the form $C=\sum_{i=1}^n\ZZ_+x_i$ for some $x_1,\ldots,x_n\in\ZZ^d$, where $C$ does not contain a positive-dimensional linear space;
\item For a cone $C\subset\RR^d$, the \emph{bottom} $\B(C)$ of $C$ is defined as the union of the compact facets of the polyhedron $\conv\big((C\cap\ZZ_+^d)\setminus\{0\}\big)\subset\RR^d$ (\cite[p. 74]{Kripo}); see Figure 1.
\item Two cones $C\subset\RR^d$ and $C'\subset\RR^{d'}$ are called \emph{lattice isomorphic} if the additive monoids $C\cap\ZZ^d$ and $C'\cap\ZZ^{d'}$ are isomoprhic; when $d=d'$ this is equivalent to the existence of an automorphism $\ZZ^d\to\ZZ^d$, inducing a linear isomorphism $C\to C'$;
\item For a cone $C\subset\RR^d$, the smallest generating set of the monoid $C\cap\ZZ^d$ is the set of indecomposable elements. It is called the \emph{Hilbert basis} of $C$, denoted by $\Hilb(C)$ and known to be finite by the \emph{Gordan Lemma} \cite[Section 2.C]{Kripo}. 

Observe, a lattice polytope $P\subset\ZZ^d$ is normal if and only if $\Hilb(\RR_+(P,1))=\{(x,1)\ |\ x\in P\cap\ZZ^d\}$.
\end{enumerate}

\begin{figure}[h!]
\caption{Bottom of a cone}
\vspace{.25in}
\includegraphics[scale=8]{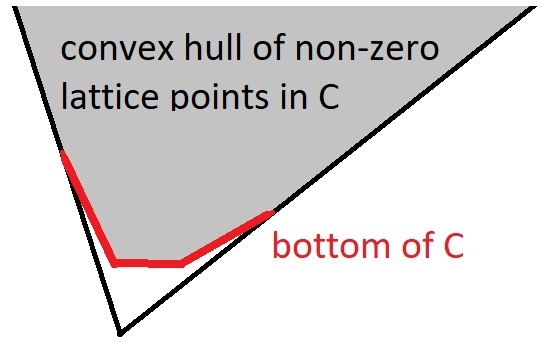}
\end{figure}

\medskip\noindent\emph{Affine monoids:}
\begin{enumerate}[\rm$\centerdot$]
\item
An \emph{affine monoid} is a finitely generated additive submonoid of a free abelian group; an affine monoid is \emph{positive} if it has no nontrivial subgroup. The monoids of the form $C\cap\ZZ^d$, where $C\subset\RR^d$ is a cone, are affine positive; conversely, an affine positive monoid $M$ is isomoprhic to $C\cap\ZZ^d$ for some cone $C\subset\RR^d$ if it is \emph{normal}, i.e., $z\in M$ whenever $z\in\gp(M)$ and $kz\in M$ for some $k\in\NN$, where $\gp(-)$ refers to the unversal group of differences. See \cite[Section 2]{Kripo} for generalities on affine monoids;
\item For a field $\kk$ and an affine monoid $M$ we think of the monoid algebra $\kk[M]$ as the corresponding monomial subalgebra of the Laurant polynomial rings $\kk[\gp(M)]\cong\kk[X_1^{\pm1},\ldots,X_d^{\pm1}]$, where $d=\rank(M)$.
\end{enumerate}

\subsection{Polytopal complexes} A (finite) \emph{polytopal complex} consists of (i) a finite family $\PPi$ of sets, (ii) a family of polytopes $P_p$, $p\in\PPi$, and (iii) a family of bijections $\pi_P:P_p\to p$, satisfying the following conditions:
\begin{enumerate}[\rm$\centerdot$]
\item For each face $F\subset P_p$, there extsts $f\in\PPi$ with $\pi_p(f)=F$,
\item For all $p,q\in\PPi$, there exist faces $F\subset P_p$ and $G\subset P_g$, such that $p\cap q=\pi_p(F)=\pi_q(G)$ and $\pi_q^{-1}\circ\pi_p:F\to G$ is an affine isomorphism of polytopes.
\end{enumerate}

\medskip For simplicity of notation, we will identify the sets $p$ and the polytopes $P_p$ along the bijections $\pi_p$. 

The \emph{support} of a polytopal complex $\PPi$ is the topological space $|\PPi|=\cup_{\PPi}P$, the elements of $\PPi$ are called \emph{faces}, and the maximal faces are called \emph{facets} of $\PPi$.

A polytopal complex $\PPi$, for which $|\PPi|$ is a polytope $P\subset\RR^d$, is called a \emph{regular subdivision} of $P$ if there is a piece-wise affine function $f:P\to\RR$, called a \emph{support function} for this subdivision, satisfying the conditions:
\begin{enumerate}[\rm$\centerdot$]
\item It is \emph{convex}, i.e., $f(\lambda x+\mu y)\le\lambda f(x)+\mu f(y)$ for all $x,y\in P$, $\lambda,\mu\in\RR_+$, $\lambda+\mu=1$, 
\item The facets of $\PPi$ are the maximal subsets of $P$ on which $f$ restricts to an affine map.
\end{enumerate}

For generalities on regular subdivisions of polytopes or, more generally, polytopal complexes, and their support functions, see \cite[Section 1]{Kripo}.

\medskip  A \emph{lattice polytopal complex} is a pair $(\PPi,{\mathbb\Lambda})$, where $\PPi$ is a polytopal complex and  $\mathbb\Lambda=\{\Lambda_P\subset\Aff(P)\ | P\in\PPi\}$ is a family of affine lattices, such that the following conditions are satisfied:
\begin{enumerate}[$\centerdot$]
\item Every polytope $P\in\PPi$ is a $\Lambda_P$-polytope,
\item $\Lambda_P\cap\Aff(P\cap Q)=\Lambda_Q\cap\Aff(P\cap Q)$ for all $P,Q\in\PPi$.
\end{enumerate}

\medskip Two lattice polytopal complexes $(\PPi,\mathbb\Lambda)$ and $(\PPi',\mathbb\Lambda')$ are \emph{isomorphic} if there is a bijection $\phi:|\PPi|\to|\PPi'|$, such that, for every $P\in\PPi$:
\begin{enumerate}[\rm$\centerdot$]
\item There exists $P'\in\PPi'$, for which $\phi|_P:P\to P'$ is an affine isomorphism, mapping $P\cap\Lambda_P$ bijectively to $P'\cap\Lambda_{P'}$,
\item By affine extension to $\Aff(P)$, the map $\phi$ induces an isomorphism $\Lambda_P\to\Lambda_{\phi(P)}$.
\end{enumerate}

We will use the notation $\phi:\PPi_1\to\PPi_2$.

\medskip A lattice complex $(\PPi,\mathbb\Lambda)$ is \emph{Euclidean} if there is an embedding $\iota:|\PPi|\to\RR^d$, such that, for every $P\in\PPi$, the restriction $\iota|_P:P\to\RR^d$ is an affine map, whose extension to $\Aff(P)$ maps $\Lambda_P$ isomorphically to $\Aff(\iota(P))\cap\ZZ^d$.

We will use the notation $\iota:\PPi\to\RR^d$.

\medskip Many examples of non-Euclidean lattice polytopal complexes, featuring various properties, are considered in \cite{Plycom}.

\medskip\noindent\emph{Convention.} If the lattice structure is clear from the context, $\mathbb\Lambda$ will be dropped from the notation.

\subsection{Polyhedral algebras}\label{Polyhedral} Let $H$ be an affine space in a vector space $V$ and $\Lambda\subset H$ be an affine lattice. To a $\Lambda$-polytope $P\subset H$ and a field $\kk$ we associate the \emph{polytopal ring} $\kk[P]$ with generators the points in $\Lambda\cap P$, subject to the binomial relations that reflect the affine dependences between these points. Alternatively, we have the monoid algebra realization $\kk[P]=\kk[M(P)]$ for the additive submonoid
$$
M(P)=\sum_{x\in P\cap\Lambda}\ZZ_+(x,1)\subset V\oplus\RR.
$$
This is a homogeneous graded algebra:
$$
\kk[P]=\kk\oplus A_1\oplus A_2\oplus\cdots=\kk[A_1],\qquad A_1=\sum_{x\in P\cap\Lambda}\kk\cdot (x,1).
$$

Observe that a lattice polytope $P\subset\RR^d$ is normal if and only if the cone $\spec(\kk[P])$ over the projective embedding $\proj(\kk[P])\to\PP_{\kk}^{N-1}$, $N=\#(P\cap\ZZ^d)$, is normal.

\medskip For a $\Lambda$-polytope $P$ and a face $F\subset P$, the \emph{face projection} $f_{PF}:\kk[P]\to\kk[F]$ is the $\kk$-algebra homomorphism, defined by
\begin{align*}
f_{PF}:(x,1)\mapsto
\begin{cases}
(x,1),\ \text{if}\ x\in F\cap\Lambda,\\
0,\ \text{if}\ x\in(P\cap\Lambda)\setminus F.
\end{cases}
\end{align*}

For a lattice polytopal complex $\PPi$, the \emph{polyhedral algebra} $\kk[\PPi]$ is defined by
\begin{align*}
\kk[\PPi]=\lim_{\longleftarrow}\big(f_{PF}:\kk[P]\to\kk[F]\ \big|\ F\subset P\ \text{a face}\big).
\end{align*} 

\medskip\noindent\emph{Notice.} The Stanley-Reisner ring of a simplicial compelx $\Delta$ is the same as the polyhedral algebra $\kk[\Delta]$, where $\Delta$ carries the coarsest lattice structure, i.e., when the only lattice points are the vertices of the simplcices.

\medskip For a lattice polytopal complex $\PPi$ and a face $P\in\PPi$, the map $M(P)\to\kk[\PPi]$ is an embedding and so we can think of $M(P)$ as a multiplicative submonoid of $\kk[\PPi]$. This way $\kk[P]$ becomes a sub-algebra of $\kk[\PPi]$. 

\medskip The elements of $\kk[\PPi]$ of the form $am$, where $a\in\kk$ and $m\in M(P)$ for some $P\in\PPi$, will be called \emph{monomials}.

Similarly to the case of a single polytope, $\kk[\PPi]$ is a homogeneous graded algebra:
$$
\kk[\PPi]=\kk\oplus A_1\oplus A_2\oplus\cdots=\kk[A_1],
$$
where $A_1$ is the $\kk$-span of the monomials of degree 1. In particular, $\kk[\PPi]$ comes equipped with the natural augmentation $\kk[\PPi]\to\kk$.

The algebra $\kk[\PPi]$ is reduced and, as a $\kk$-vector space, it equals the inductive limit of the diagram of face embeddings:
\begin{align*}
\kk[\PPi]=\lim_{\longrightarrow}\big(\kk[F]\hookrightarrow\kk[P]\ \big|\ F\subset P\ \text{a face}\big)
\end{align*}

Homological properties of polytopal algebras $\kk[\PPi]$ for $\PPi$ Euclidean are studied in \cite{Stanley}; the linear group of all graded automorphisms of the algebra $\kk[\PPi]$ for $\PPi$ not necessarily Euclidean are studied in \cite{Plycom}. 

As it turns out, an Euclidean lattice polytopal complex $\PPi$ is uniquely determined by its polyhedal algebra $\kk[\PPi]$; in particular, we recover the old result that a simplicial complex $\Delta$ is determined by its Stanley-Reisner ring $\kk[\Delta]$:

\begin{theorem}\label{isomorphism}
For a field $\kk$ and two Euclidean lattice polytopal complexes $\PPi,\PPi'$, the algebras $\kk[\PPi]$ and $\kk[\PPi']$ are isomorphic as augmented $\kk$-algebras if and only if $\PPi\cong\PPi\rq{}$.
\end{theorem}

\begin{proof} (Compare with \cite[Proposition 5.26]{Kripo}, which yields the result for a single polytope.)
Assume $\kk[\PPi]\cong\kk[\PPi']$ as augmented algebras. By scalar extension, we can assume $\kk=\bar\kk$, the algebraic closure of $\kk$. Passing to the associated graded isomorphism with respect to the augmentation ideals, we derive a graded isomorphism $f:\kk[\PPi]=\gr(\kk[\PPi])\to\gr(\kk[\PPi'])=\kk[\PPi']$. We can identify $\kk[\PPi]$ and $\kk[\PPi']$ along $f$. Let $\mathbb T$ denote the unity component of the linear group of \emph{diagonal automorphisms} of $\kk[\PPi]$ with respect to $\PPi$, i.e., of the automorphisms, for which every $\PPi$-monomial is an eigenvector, and similarly for $\mathbb T'$. According to \cite[Lemma 3.5(b)]{Plycom}, both groups $\mathbb T$ and $\mathbb T'$ are maximal tori in the linear group $\Gamma_\kk(\PPi)$ of all graded atomorphisms of $\kk[\PPi]$. By Borel's Theorem on maximal tori \cite[Corollary 11.3]{Borel}, $\mathbb T$ and $\mathbb T'$ are conjugate in $\Gamma_k(\PPi)$. Assume $\mathbb T'=\gamma^{-1}\mathbb T\gamma$. Then the automorphism $\gamma:\kk[\PPi]\to\kk[\PPi]$ maps $\mathbb T'$-eigenvectors to $\mathbb T$-eigenvectors. But it follows from \cite[Lemma 3.5(a)]{Plycom} that $\mathbb T'$-eigenvectors are the $\PPi'$-\emph{monomials} and $\mathbb T$-eigenvectors are the $\PPi$-\emph{monomials}. In particular, the quotients of the multiplicative monoids of $\PPi$- and $\PPi'$-monomials by the multiplicative group $\kk^*$ are isomorphic. From this one easily derives $\PPi\cong\PPi'$.
\end{proof}

\section{Rees deformations and bottom complexes}\label{Commutative_algebra}

\subsection{Rees deformations of toric rings}
For a field $\kk$, a $\kk$-algebra $R$, and an ideal $I\subset R$, the \emph{Rees algebra} $(\mathcal R,I)=R[T,T^{-1}I]\subset R[T,T^{-1}]$ is flat over $\kk[T]$; moreover, the fiber of the map $\spec(\mathcal R,I)\to\AA^1_\kk$ at $T=a\in\kk\setminus\{0\}$ is isomorphic to $\spec(R)$, whereas the fiber at $T=0$ is isomorphic to $\spec(\gr_I(R))$ for the \emph{associated graded ring} $\gr_I(R)=R/I\oplus R/I^2\oplus\cdots$ (\cite[Section 6.5]{Eisenbud}). In particular, and especially when $\kk$ is algebraically closed, $\gr_I(R)$ can be regarded as a flat deformation of $R$.

Rees' deformations of rings of the form $\kk[C\cap\ZZ^d]$, where $C\subset\RR^d$ is a cone, with respect to the maximal monomial ideals have a nice description in terms of \emph{bottom complexes}. We denote by $\gr(\kk[C\cap\ZZ^d])$ the mentioned associated graded ring.

\begin{definition}\label{BOTTOMcomplexes}
\begin{enumerate}[\rm(a)] 
\item For a cone $C\subset\ZZ^d$, its \emph{bottom complex} $\BB(C)$ is the Euclidean lattice polytopal complex, whose facets are the maximal polytopes in $\B(C)$ and the affine lattices of reference are induced from $\ZZ^d$.
\item A lattice polytopal complex $\PPi$ is \emph{bottom} if there is a cone $C$, such that $\PPi\cong\BB(C)$. Such a cone $C$ is called a \emph{conic realization} of $\PPi$.
\item A lattice polytopal complex $\PPi$ is \emph{reduced bottom} if there is a cone $C\subset\RR^d$, such that $(\PPi,\mathbb\Lambda)\cong\BB(C)$ and, for every facet 
$F\in\BB(C)$, the subcone $\RR_+F\subset C$ satisfies $\Hilb(\RR_+F)\subset F$. Such $C$ is called a \emph{reduced conic realization} of $\PPi$.
\item Two (reduced) bottom realizations of a polytopal complex are called \emph{isomorphic} if the two cones are lattice isomorphic.
\end{enumerate}
\end{definition}

\medskip\noindent\emph{Notice.} The condition on the Hilbert bases in Definition \ref{BOTTOMcomplexes}(c) seems to be stronger than the containment $\Hilb(C)\subset\B(C)$, but we do not have an appropriate example.

\medskip A union of finitely many polytopes that is homeomorphic to a closed ball is called a \emph{polytopal ball}. We say that a polytopal $d$-ball $\B\subset\RR^{d+1}$ is \emph{convex towards 0} if the following conditions are satisfied:
\begin{enumerate}[\rm$\centerdot$]
\item $0\notin\B$ and the conical set $\RR_+\B\subset\RR^{d+1}$ is a $(d+1)$-cone, 
\item For every maximal polytope $P\subset\B$, the point $0$ and $\B\setminus P$ are in different open affine half-spaces, separated by $\Aff(P)$. (In particular, for every $x\in\B$, the equality $\RR_+x\cap\B=\{x\}$ holds.) 
\end{enumerate}

The proof of the following alternative description of bottom complexes, useful in the next sections, is straightforward:

\begin{lemma}\label{alternative}
Let $\PPi$ be a $d$-dimensional lattice polytopal complex.
\begin{enumerate}[\rm(a)]
\item $\PPi$ is bottom if and only if there exists an embedding $\iota:\PPi\to\RR^{d+1}$, such that:
\begin{enumerate}[\rm(i)]
\item $|\iota(\PPi)|)$ a polytopal $d$-ball, convex towards $0$,
\item For every facet $P\in\PP$, there are no lattice points in $\conv(0,|\iota(P)|)$ except $0$ and $|\iota(P)|\cap\ZZ^{d+1}$.
\end{enumerate}
\item $\PPi$ is reduced bottom if and only if there exists an embedding $\iota:\PPi\to\RR^{d+1}$, such that:
\begin{enumerate}[\rm(i)]
\item $|\iota(\PPi)|$ a polytopal $d$-ball, convex towards $0$,
\item For every facet $P\in\PPi$, the polytope $\iota(P)$ is normal,
\item For every facet $P\in\PPi$, the affine hyperplane $\Aff(\iota(P))$ is on lattice distance one from $0$ with respect to the standard lattice $\ZZ^d$.  
\end{enumerate}
\end{enumerate}
\end{lemma} 

\begin{theorem}\label{Folklore}
For a cone $C\subset\RR^d$, the ring $\kk[\BB(C)]$ embeds in $\gr(\kk[C\cap\ZZ^d])$ as a $\kk$-algebra retract and the kernel of the $\kk$-retraction $\gr(\kk[C\cap\ZZ^d])\to\kk[\BB(C)]$ is the nilradical.
\end{theorem}

\begin{proof}
First we consider the case when $\BB(C)$ is (the boundary complex of) a single $(d-1)$-dimensional lattice polytope $P$. The $c$-th graded component $I^c/I^{c+1}\subset\gr(\kk[C\cap\ZZ^d])$ is the $\kk$-linear span of the monomials $m\in C\cap\ZZ^d$, which are products of $c$ monomials of positive degree under the grading
$$
\kk[C\cap\ZZ^d]=\kk\oplus B_1\oplus B_2\oplus\cdots,\qquad B_c=\sum_{x\in C\cap(\ZZ^{d-1},c)}\kk\cdot x,
$$
but cannot be written as products of more than $c$ monomials of positive degree. We denote the corresponding degrees by $\deg(-)$. In particular, we have the identification of $\kk$-vector spaces $\kk[C\cap\ZZ^d]=\gr(\kk[C\cap\ZZ^d])$. Also, we have the sub-algebra 
$$
\kk[P]\cong\kk[\BB(C)]=\kk[B_1]\subset\kk[C\cap\ZZ^d].
$$ 

The multiplicative structures of $\gr(\kk[C\cap\ZZ^d])$ is described as follows. For every element $m\in (C\cap\ZZ^d)\setminus\{0\}$, let $\l(m)$ denote the maximal decomposition length of $m$ in the monoid $C\cap\ZZ^d$. Equivalently, $\l(m)$ is the degree of the corresponding element in the graded ring
$\gr(\kk[C\cap\ZZ^d])$.

For two elements $m_1,m_2\in(C\cap\ZZ^d)\setminus\{0\}$, their product $m_1\cdot m_2$ in $\gr(\kk[C\cap\ZZ^d])$ is given by 
\begin{align*}
m_1\cdot m_2=
\begin{cases}
m_1m_2\in\kk[C\cap\ZZ^d]\ \text{if}\ \l(m_1m_2)=\l(m_1)+\l(m_2),\\
0\ \text{if}\ \l(m_1m_2)>\l(m_1)+\l(m_2).\\
\end{cases}
\end{align*}

Observe that, for every element $m\in(C\cap\ZZ^d)\setminus\{0\}$, we have $\l(m)\le\deg(m)$ with the equality if and only if $m\in\kk[P]$. Simultaneously, there exists $t\in\NN$, such that $m^t\in\kk[\BB(C)]$. This implies that, on the one hand, the product of any system of monomials in $\kk[\BB(C)]$ is the same as that of these monomials in $\gr(\kk[C\cap\ZZ^d])$ and, on the other hand, every monomial in $\kk[C\cap\ZZ^d]\setminus\kk[\BB(C)]$ represents a nilpotent element in $\gr(\kk[C\cap\ZZ^d])$. This proves the special case of Theorem \ref{Folklore} when $\BB(C)$ has a single facet.

Next we observe, that for a general cone $C$, the similar identification of $\kk$-vector spaces $\kk[C\cap\ZZ^d]=\gr(\kk(C\cap\ZZ^d])$ still makes sense. On the other hand, for a facet $F\in\BB(C)$, there is a grading
$$
\kk[C\cap\ZZ^d]=\kk\oplus B_1\oplus B_2\oplus\cdots,
$$
called the \emph{basic grading} with respect to $F$ in \cite[p. 74]{Kripo}, making the monomials homogeneous and such that the resulting degree $\deg_F(-)$ satisfies $\deg_F(m_1)<\deg_F(m_2)$ for any elements  $m_1\in F\cap\ZZ^d$ and $m_2\in(C\cap\ZZ^d)\setminus\RR_+F$. This grading implies that no nonzero element $m\in \sum_{x\in F\cap\ZZ^d}\ZZ_+x$ can be decomposed within the bigger monoid $C\cap\ZZ^d$ into more than $\deg_F(m)$ elements. In particular, the identity embedding $\kk[F]\to\gr(\kk[C\cap\ZZ^d])$ respects the multiplicative structure, i.e., $\kk[F]$ is a subalgebra of $\gr(\kk[C\cap\ZZ^d])$ and the general case reduces to the case when $\BB(C)$ has only one facet.
\end{proof}

Theorems \ref{isomorphism} and \ref{Folklore} have the following

\begin{corollary}\label{bottom_algebra}
\begin{enumerate}[\rm(a)]
\item A $d$-dimensional lattice polytopal complex $\PPi$ is bottom if an only if there is a cone $C\subset\RR^{d+1}$, such that $\kk[\PPi]\cong\gr(\kk[C\cap\ZZ^d])_{\text{\emph{red}}}$ as augmented $\kk$-algebras.
\item
A $d$-dimensional lattice polytopal complex $\PPi$ is reduced bottom if and only if there is a cone $C\subset\RR^{d+1}$, such that $\kk[\PPi]\cong\gr(\kk[C\cap\ZZ^d])$ as augmented $\kk$-algebras. 
\end{enumerate}
\end{corollary}

\section{Gluing bottom complexes}\label{REALIZATION}

\begin{lemma}[{\bf Cone gluing}]\label{Patching}
Let $C_i\subset\RR^d$ be $d$-cones and $F_i\subset C_i$ be facets for $i=1,2$. Assume $F_1$ and $F_2$ are lattice isomorphic. Then there exists a cone $C\subset\RR^d$ and a rational linear form $h:\RR^d\to\RR$, together with lattice isomorphisms:
\begin{align*}
C^0\cong F_1\cong F_2,\quad C_1\cong C^+,\quad C_2\cong C^-,
\end{align*}
where $C^0=\{x\in C\ |\ h(x)=0\}$, $C^+=\{x\in C\ |\ h(x)\ge0\}$, and $C^-=\{x\in C\ |\ h(x)\le0\}$.
\end{lemma}

\begin{proof}
Pick an isomorphism $F_1\cap\ZZ^d\to F_2\cap\ZZ^d$ and extend it to an $\RR$-automorphism $\theta:\RR^d\to\RR^d$, which maps $\ZZ^d$ bijectively to itself and maps the cones $C_1$ and $C_2$ to the opposite sides relative to the hyperplane $\RR F_2$. Put $D_i=\theta(C_i)$, $i=1,2$. We have:

\begin{enumerate}[\rm$\centerdot$]
\item $D_1$ and $D_2$ share the facet $F_2$ and they are in the opposite sides relative to $\RR F_2$,
\item $D_i$ and $C_i$ are lattice isomorphic, $i=1,2$.
\end{enumerate}

If the subset $D_1\cup D_2\subset\RR^d$ is a cone, which is equivalent to $D_1\cup D_2$ being convex, then the cone $C=D_1\cup D_2$ satisfies the desired condition. In general, $D_1\cup D_2$ is not a cone. We fix this as follows. Pick two bases of $\ZZ^d$ of the form 
\begin{align*}
\mathcal B_1=\{u,w_2,\ldots,w_d\},\qquad\mathcal B_2=\{v,w_2,\ldots,w_d\},
\end{align*}
where
\begin{enumerate}[\rm$\centerdot$]
\item $\{w_2,\ldots,w_d\}\subset\RR F_2$,
\item $u$ is on the same side relative to $\RR F_2$ as $D_1$,
\item $v$ is on the same side relative to $\RR F_2$ as $D_2$.  
\end{enumerate}

Pick an element $\gamma\in\big(\inte(F_2)\cap\ZZ^d\big)\setminus\{0\}$. For a natural number $t$, consider the following bases of $\ZZ^d$:
\begin{align*}
\mathcal B_1(t)=\{u+t\gamma,w_2,\ldots,w_d\},\qquad \mathcal B_2(t)=\{v+t\gamma,w_2,\ldots,w_d\}
\end{align*}
and the automorphisms:
\begin{tabbing}
\hspace{2.8in}\= \hspace{0in}\kill
$\alpha_t:\ZZ^d\to\ZZ^d$, \>  $\beta_t:\ZZ^d\to\ZZ^d$,\\
\quad\ $\lambda_1u+\lambda_2w_2+\cdots+\lambda_dw_d\ \mapsto$ \>\ \quad$\mu_1v+\mu_2w_2+\cdots+\mu_dw_d\ \mapsto$\\
\quad\ $\lambda_1(u+t\gamma)+\lambda_2w_2+\cdots+\lambda_dw_d$, \> $\quad\ \mu_1(v+t\gamma)+\mu_2w_2+\cdots+\mu_dw_d$.
\end{tabbing}

\medskip\noindent\emph{Notice.} The maps $\alpha_t$ and $\beta_t$ are independent of the choices  of $u$ and $v$. In fact, if $u'$ and $v'$ satisfy the same conditions as $u$ and $v$, respectively, then an equality in $\ZZ^d$
\begin{align*}
\lambda_1 u+\lambda_2 w_2+\cdots+\lambda_d w_d=\lambda'_1 u'+\lambda'_2 w_2+\cdots+\lambda'_d w_d
\end{align*}
forces $\lambda_1=\lambda'_1$, and similarly for $v'$.

\medskip We claim that for $t$ large, the set $C_t=\alpha_t(D_1)\cup\beta_t(D_2)$ is a cone in $\RR^d$ which, along with the subcones
\begin{align*}
C^+_t:=\alpha_t(D_1),\qquad C^-_t:=\beta_t(D_2),
\end{align*}
and an appropriate linear map $h:\RR^d\to\RR$ with $\ker h=\RR F_2$ has the desired properties.

Let $\{x_1,\ldots,x_m\}$ and $\{y_1,\ldots,y_n\}$ be arbitrary non-zero vectors in $\ZZ^d$, representing the extremal rays of $D_1$ and $D_2$, respectively, that are not in $F_2$; we pick one vector per a ray. Also, let $\{z_1,\ldots,z_l\}$ be non-zero vectors in $\ZZ^d$, representing the shared extremal rays of $D_1$ and $D_2$, i.e., the extremal rays of $F_2$.

We have
\begin{align*}
&C^+_t=\RR_+\alpha_t(x_1)+\cdots+\RR_+\alpha_t(x_N)+\RR_+z_1+\cdots+\RR_+z_l,\\
&C^-_t=\RR_+\beta_t(y_1)+\cdots+\RR_+\beta_t(y_N)+\RR_+z_1+\cdots+\RR_+z_l.\\
\end{align*}

As $t\to\infty$, the radial directions of the points $\alpha_t(x_i)$ converges to that of $\gamma$. The same is true for the radial directions of $\beta_t(y_j)$. Since $\gamma\in\inte(F)$, for every pair of indices $i,j$ and a sufficiently large natural number $t$, the segment $[\alpha_t(x_i),\beta_t(y_j)]$ meets $\inte(F_2)$. But then, for any two non-zero points $x\in\RR_+\alpha_t(x_i)$ and $y\in\RR_+\beta(y_j)$, the segment $[x,y]$ meets $\inte(F_2)$. This is equivalent to $C_t$ being a cone.
\end{proof}

Assume two $d$-dimensional (reduced) bottom lattice complexes $\PPi_1$ and $\PPi_2$ admit (reduced) conic realizations $C_1,C_2\subset\RR^{d+1}$ and assume there is a lattice isomorphism $\theta:F_1\to F_2$ for some facets $F_1\subset C_1$ and $F_2\subset C_2$. Then, identifying the appropriate pairs of facets of $\PPi_1$ and $\PPi_2$ along $\theta$, we can define a new lattice polytopal complex -- the  \emph{conic gluing of $\PPi_1$ and $\PPi_2$ along $\theta$,} denoted by $\PPi_1\vee_\theta\PPi_2$, or just $\PPi_1\vee\PPi_2$ when $\theta$ is clear from  the context.

More systematically, first one chooses cones $D_1$ and $D_2$ as in the proof of  Lemma \ref{Patching}. Observe that, for every cone, the bottom complex of a face is a sub-complex of the bottom of the cone. In particular, the faces of $\PPi_1$, whose conic realizations in $D_1$ are in the common face $D_1\cap D_2$, coincide with the conic realizations in $D_2$ of the appropriate faces of $\PPi_2$. In other words, $\B(D_1)\cup\B(D_2)$ and the standard lattice $\ZZ^{d+1}\subset\RR^{d+1}$ define an embedding 
\begin{equation}\label{special_emb}
\PPi_1\vee_\theta\PPi_2\to\RR^{d+1}.
\end{equation}

Call a lattice polytopal complex $\PPi$ \emph{stacked} if its facets can be enumerated in such a way $P_1,\ldots,P_n$ that $P_i\cap(P_1\cup\ldots\cup P_{i-1})$ is a facet of $P_i$ for every $i=2,\ldots,n$.

\begin{figure}[h!]
\caption{Stacked polytopal complex}
\vspace{.25in}
\includegraphics[scale=8]{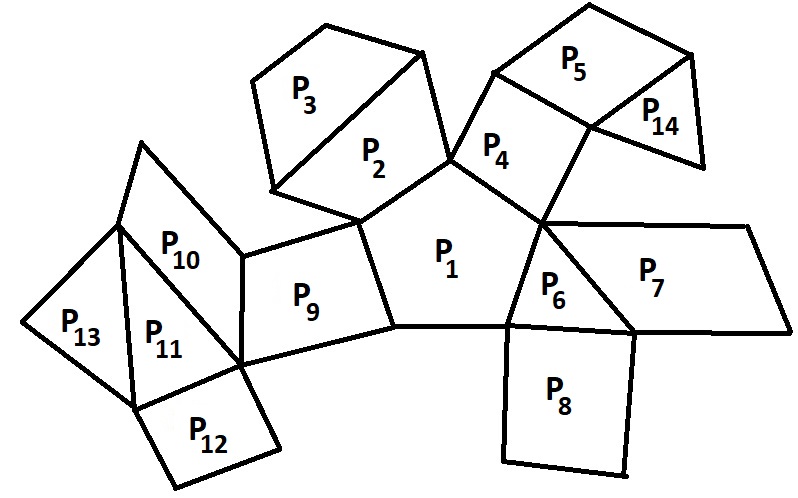}
\end{figure}

\begin{theorem}\label{MAIN}
\begin{enumerate}[\rm(a)]
\item
Assume $\PPi_1$ and $\PPi_2$ are (reduced) bottom complexes that can be conically glued. Then $\PPi_1\vee\PPi_2$ is (reduced) bottom.
\item
Every stacked lattice polytopal complex is bottom; it is reduced if and only if the facets of the complex are  normal polytopes with respect to the lattices of reference. 
\end{enumerate}
\end{theorem}

\begin{proof} (a) Assume $\PPi_1$ and $\PPi_2$ can be conically glued. We think of $\PPi_1\vee\PPi_2$ as an embedded lattice complex via $(\ref{special_emb})$. Using the notation above, in view of Lemma \ref{alternative}, one only needs to achieve that the polytopal $d$-ball $\B(D_1)\cup\B(D_2)\subset\RR^{d+1}$ is convex towards $0$. Using the notation in the proof of Lemma \ref{Patching}, this can be done by changing $D_1$ and $D_2$ to $\alpha_t(D_1)$ and $\beta_t(D_2)$, respectively, for $t$ sufficiently large. In fact, not only becomes the union $\alpha_t(D_1)\cup\beta(D_2)$ a cone as $t\to\infty$, but the polytopal $d$-ball $\B(\alpha_t(D_1))\cup\B(\beta_t(D_2))$ also becomes convex towards $0$ for $t$ large: the convexity towards $0$ may only fail along $\B(\alpha_t(D_1))\cap\B(\beta_t(D_2))$, but  as $t\to\infty$ the two polytopal balls fold away from $0$ leaving $\B(\alpha_t(D_1))\cap\B(\beta_t(D_2))$ fixed.

\medskip\noindent(b) This follows from the part (a) by induction on $n$. In fact, assume the facets of a stacked lattice complex $\PPi$ are $P_1,\ldots,P_n$, enumerated this way way. Let $\PPi_i\subset\PPi$ denote the lattice polytopal sub-complex with $|\PPi_i|=P_1\cup\ldots\cup P_i$. If $C_i$ is a conic realization of $\PPi_i$ for some $1\le i< n$ then the corresponding facet of $C_i$ is lattice isomorphic to the facet $\RR_+\big(P_{i+1}\cap(P_1\cup\ldots\cup P_i)\big)\subset\RR_+P_{i+1}$.
\end{proof}

\begin{remark}\label{Shelling}
Theorem \ref{MAIN}(a) leads to a much larger class of (reduced) bottom complexes than the class in Theorem \ref{MAIN}(b). Consider a sequence of (reduced) bottom complexes $\PPi_1,\ldots,\PPi_n$, such that, for every $1\le i<n$, the inductively defined complex $\vee_{k=1}^i\PPi_k$ and $\PPi_{i+1}$ can be conically glued. Then $\vee_{i=1}^n\PPi_i$ is also (reduced) bottom. The complex $\vee_{i=1}^n\PPi_i$ carries a shellable-like structure (\cite[Section 15]{Stacked}), which is more general than stacking polytopes inductively, reminiscent of the definition of a \emph{stacked polytope} (\cite[Section 19]{Stacked}). More precisely, many reduced bottom complexes, not covered by Theorem \ref{MAIN}(b), are described in Sections \ref{High_D} and \ref{Bottom_discs}. Incorporating such non-stacked building blocks in the conical gluing process, one produces a considerably larger class of reduced bottom complexes than the stacked ones.   
\end{remark}

\begin{remark}\label{1-dimensional}\label{One-dimensional} It follows from Theorem \ref{MAIN}(b) that every 1-dimensional lattice polytopal complex $\PPi$, whose support is homeomorphic to an interval, is reduced bottom. More precisely, if the lattice points of $\PPi$ are labeled successively by $1,2,\ldots,n$ for some $n\ge2$, then the corresponding lattice points $x_1,x_2,\ldots,x_n$ in a reduced conic realization are subject to relations of the form $x_{k-1}+x_{k+1}=c_kx_k$ for some $c_2,\ldots,c_{n-1}\in\ZZ_{\ge2}$ \cite[Section 1.6]{Oda}. This gives rise to a bijection between the isomorphism classes of reduced conic realizations of $\PPi$ and the set $(\ZZ_{\ge2})^{n-2}/\ZZ_2$, where $\ZZ_2$ acts on the $(n-2)$-tuples by inversion. Furthermore, this correspondence restricts to a bijection between the isomorphism classes of reduced conic realizations of 1-dimensional simplicial complexes with $n$ vertices, topologically equivalent to an interval, and the set $(\ZZ_{\ge3})^{n-2}/\ZZ_2$.

\begin{remark}\label{0-dimensional}
Although the case of a zero-dimensional polytopal complex is trivial, an interesting question in a complementary direction in the zero-dimensional case is to study $\gr_I(\kk[M])$ for $M$ a \emph{numerical monoid} and $I\subset\kk[M]$ the maximal monomial ideal \cite{Olinger}. Recall, a  submonoid $M\subset\ZZ_+$ is called \emph{numerical} if $\#(\ZZ_+\setminus M)<\infty$.   
\end{remark}
\end{remark}

\section{Bottom simplicial balls}\label{High_D}

In this and the next sections abstract simplicial complexes are considered as lattice polytopal complexes with respect to the coarsest lattice structure. Equivalently, the simplices in abstract simplicial complexes are considered to be unimodular.

A \emph{simplicial sphere} is a simplicial complex, whose geometric realization is homeomorphic to a sphere. For a simplicial sphere $\Sigma$, a simplicial ball, obtained by taking the pyramids with a common apex over the faces of $\Sigma$, will be denoted by $\Delta(\Sigma)$.

\subsection{Obstructions to bottom}\label{Obstructions} 

\medskip\noindent{\bf\emph{Geometric obstruction.}} Obviously, a necessary condition for an abstract simplicial complex $\Delta$ to be bottom is that it needs to be a \emph{simplicial ball,} i.e., $|\Delta|$ must be homeomorphic to a $d$-ball, $d=\dim\Delta$. But this is far from sufficient.

Call a simplicial ball \emph{regular} if it is combinatorially equivalent to a regular triangulation of a polytope. A bottom simplicial $d$-ball $\Delta$ is necessarily regular. In fact, if $C\subset\RR^{d+1}$ is a conic realization of $\Delta$ and $H\subset\RR^{d+1}$ is a hyperplane with $C\cap H=0$ then the projective transformation of $\RR^{d+1}$, moving a $H$ to infinity, transforms $C$ into an infinite prism and the bottom $\B(C)$ into the graph of a support function for a regular triangulation of the orthogonal cross-section of this prism. But this triangulation is equivalent to $\Delta$.

A simplicial sphere is \emph{polytopal} if it is combinatorially equivalent to the boundary complex of a \emph{simplicial polytope,} i.e., a polytope, whose faces are all simplices. The smallest non-polytopal simplcial sphere has dimension 3 and, starting from dimension 5, the class of polytopal simplicial spheres is negligibly small withing all simplicial spheres; see \cite[Section 9.5]{Triang} and the many original references therein.  This fact and the observation above lead to many examples of non-bottom simplicial balls in dimensions $\ge4$: if $\Sigma$ is a non-polytopal simplicial $(d-1)$-sphere then $\Delta(\Sigma)$ is not bottom.

\medskip\noindent{\bf\emph{Lattice obstruction.}} Even if a simplicial sphere $\Sigma$ is polytopal the simplicial ball $\Delta(\Sigma)$ may still fail to be bottom reduced.  In fact, it is shown in \cite{Klein} that the simplicial complex $\Delta(\Sigma_k)$, where $\Sigma_k$ is the boundary complex of a cyclic $4$-polytope with $k\ge 7$ vertices, is not combinatorially equivalent to a simplicial complex, embedded in $\RR^4$ as a system of unimodular simplices. But then neither is $\Delta(\Sigma_k)$ reduced bottom. In fact, assume $C\subset\RR^5$ is a reduced conic realization. Let $H\subset\RR^5$ be a hyperplane with $H\cap C=0$. Let $O\in C$ be the point, corresponding to the interior vertex of $\Delta(\Sigma_k)$. Denote by $\pr:\RR^5\to H$ the parallel projection along $\RR O$. Then the (not necessarily convex) set $\pr(\B(C))$ is triangulated into unimodular 4-simplices with respect to the lattice $\Lambda=\pr(\ZZ^5)\subset H$, a consequence of the fact that these simplices are of the form $\conv(0,\pr(v_1),\ldots,\pr(v_4))$ with $\{v,v_1,\ldots,v_4\}$ a basis of $\ZZ^5$; i.e., $\Delta(\Sigma_k)$ embeds into $H$ as a system of $\Lambda$-unimodular simplices, a contradiction.

\subsection{Bottom complexes from smooth Fano polytopes.}\label{Smooth_Fano} A $d$-dimensional \emph{smooth Fano polytope} is a lattice $d$-polytope $P$, containing $0$ in the interior and such that the vertices of every facet $F\subset P$ form a basis of $\ZZ^d$. This is equivalent to the condition that the toric variety, corresponding to the complete fan of cones over the faces of $P$, is smooth and Fano, i.e., the anticanonical bundle is ample; see \cite[Section 5.8]{Ewald}, \cite[Section 2.3]{Oda}. In every dimension $d$, up to unimodular equivalence, there are only finitely many smooth Fano $d$-polytopes and their complete classification is only known in low dimensions; see \cite{Fano_Nill} for the original references and many applications of these polytopes. In view of Lemma \ref{alternative}(b), reduced bottom simplicial complexes can be thought of as dual objects to smooth Fano polytopes, the duality being `convex away from the origin' vs. `convex towards the origin'. (There is no bottom counterpart of `complete'). 

\medskip There is another and more direct connection between smooth Fano polytopes and reduced bottom simplicial complexes, which we discuss now. 

Let $\Sigma$ be a simplicial $(d-1)$-sphere. Assume $\Delta(\Sigma)$ is reduced bottom, notation as in Section \ref{Obstructions}. Call a reduced conic realization $C\subset\RR^{d+1}$  of $\Delta(\Sigma)$ \emph{regular} if the following set is convex, i.e., is an infinite convex prism:
\begin{align*}
\bigcup_{x\in\B(C)}(x+\RR O)\subset\RR^{d+1},
\end{align*}
where $O$ is the interior vertex of $\B(C)$. Our goal in this sectoin is to describe all regular reduced conic realizations of $\Delta(\Sigma)$.

\medskip Let $P\subset\RR^d$ be a simplicial $d$-polytope with $0\in\inte(P)$. Assume $v_1,\ldots,v_n$ are the vertices of $P$ and $\conv(v_1,\ldots,v_d)\subset P$ is a facet.

To every pair of adjacent facets of $P$ 
\begin{align*}
F=\conv\big(v_{i_1},\ldots,v_{i_{d-1}},v_{i_d}\big)\quad\text{\&}\quad G=\conv(v_{i_1},\ldots,v_{i_{d-1}},v_{i_{d+1}})
\end{align*}
we associate the functional
\begin{equation}
\begin{aligned}\label{functional}
&\phi_{FG}:\RR^{n+1}\to\RR,\\
&(x_0,x_1,\ldots,x_n)\mapsto\lambda_{i_{d+1}}x_{i_{d+1}}+\lambda_{i_d}x_{i_d}-\lambda_{i_{d-1}}x_{i_{d-1}}-\cdots-\lambda_{i_1}x_{i_1}-\lambda_0x_0,
\end{aligned}
\end{equation}
where the $\lambda_i$ are uniquely determined by the conditions:
\begin{equation}\label{baricentic}
\begin{aligned}
&\lambda_{i_d}v_{i_d}+\lambda_{i_{d+1}}v_{i_{d+1}}=\lambda_{i_1}v_{i_1}+\cdots+\lambda_{i_{d-1}}v_{i_{d-1}},\\
&\lambda_{i_d}+\lambda_{i_{d+1}}=1,\\
&\lambda_0+\lambda_{i_1}+\cdots+\lambda_{i_{d-1}}=1,\\
&\lambda_0,\ \lambda_{i_d},\ \lambda_{i_{d+1}}>0,
\end{aligned}
\end{equation}
the inequalities being automatic from the three equalities.

\medskip\noindent\emph{Notice.} If $P$ is a smooth Fano polytope then $\lambda_{i_d}=\lambda_{i_{d+1}}=\frac12$.

\medskip Consider the following convex conical set in $\RR^{n+1}$:
\begin{align*}
\SF(P)=\big\{(x_0,x_1,\ldots,x_n)\ \ \big|\ &f_{FG}(x_0,x_1,\ldots,x_n)\ge0\ \text{for every adjacent}\\
&\text{pair of facets}\ F,G\subset P\big\}.
\end{align*}

Every point $\bar x=(x_0,x_1,\ldots,x_n)\in\RR^{n+1}$ defines the polytopal $d$-ball:
\begin{align*}
\B(\bar x)=\bigcup_
{\tiny{\begin{matrix}
&\conv\big(v_{i_1},\ldots,v_{i_d}\big)\\
&\text{a facet of}\ P
&\end{matrix}
}}
\conv\big((0,x_0),(v_{i_1},x_1),\ldots,(v_{i_{d}},x_d)\big)\subset P\times\RR
\end{align*}

The crucial observation is that the points $\bar x\in\inte(\SF(P))$ are in bijective correspondence with the support functions $P\to\RR$ for the stellar triangulation of $P$, spanned over $0$ by the faces of $P$: to a point $\bar x$ one assigns the function, whose graph is $\B(\bar x)$. This also explains the notation $\SF$.

Consider the sets
\begin{align*}
&\SF^+(P)=\inte(\SF(P))\cap\big(X_0=0,\ X_1,\ldots,X_n\ge0\big),\\
&\SF^0(P)=\inte(\SF(P))\cap\big(X_0=X_1=\cdots=X_d=0,\ X_{d+1},\ldots,X_n\ge0\big).
\end{align*}

\medskip\noindent\emph{Notice.} In the definition of $\SF^0(P)$, we could equivalently require $X_{d+1},\ldots,X_n>0$.

\medskip Let $\Aut(P)$ be the group of linear automoprhisms of $P$. It acts on $P\times\RR_+$ by fixing the $(d+1)$-st coordinate  which, in turn, defines an action $\Aut(P)$ on $\SF^+(P)$ by linear automorphisms as follows: for $\alpha\in\Aut(P)$ and $x\in\SF^+(P)$, the element $\alpha(\bar x)$ is determined from the equality $\B(\alpha(\bar x)=\alpha(\B(\bar x))$, i.e., $\alpha$ permutes the coordinates of $\bar x$ appropriately.

In the notation above, we have

\begin{lemma}\label{combined}
\begin{enumerate}[\rm(a)]
\item $\dim(\SF(P)^0)=n-d$.
\item For every point $\bar x\in\SF^+(P)$, the following conical subset is a cone:
\begin{align*}
C(P,\bar x):=\RR_+(\B(\bar x)+\ee_{d+1})\subset\RR^{d+1}\qquad\text{(see Figure 3)}.
\end{align*}
\item For every point $\bar x\in\SF^+(P)$, there is a unique point $\rho(\bar x)\in\SF^0(P)$, such that the cone $C(P,\rho(\bar x))$ is obtained from $C(P,\bar x)$ by a linear transformation $\phi:\RR^{d+1}\to\RR^{d+1}$ with $\Im({\bf1}-\phi)\subset\RR\ee_{d+1}$.
\item The map $\rho:\SF^+(P)\to\SF^0(P)$, resulting from the part (c), is linear.
\item
$\Aut(P)$-acts on $\SF^0(P)$ by linear automorphisms via
\begin{align*}
\forall\alpha\in\Aut(P),\quad\forall\bar x\in\SF^0(P),\quad\alpha\star\bar x=\rho(\alpha(\bar x)).
\end{align*}
\end{enumerate}
\end{lemma}

\begin{figure}[h!]
\caption{Cone $C(P,\bar x)$ for $x\in\SF^0(P)$}
\vspace{.25in}
\includegraphics[scale=11]{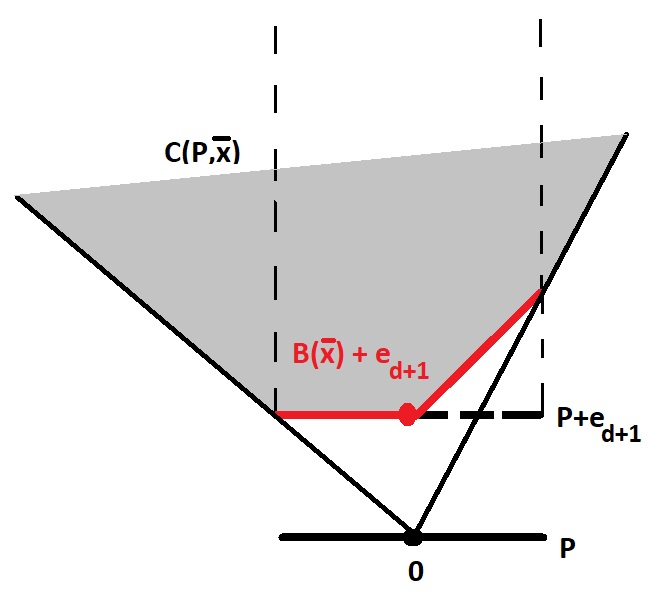}
\end{figure}

\begin{proof} (a) We only need to show that the dimension in question is $\ge n-d$. For every $\bar x=(0,\ldots,0,x_{d+1},\ldots,x_n)$ in the mentioned cone, the polytopal ball $\B(\bar x)$ defines a support function for the stellar subdivision of $P$. So all small random perturbations of $\bar x$ of the form $(0,\ldots,0,x'_{d+1},\ldots,x'_n)$ still define support functions for this triangulation.

\medskip\noindent(b) Let $\bar x=(0,x_1,\ldots,x_n)\in\SF^+(P)$. The convexity of $C(P,\bar x)$ is equivalent to the inequalities $\phi_{FG}(0,x_1+1,\ldots,x_n+1)>0$
for the functionals (\ref{functional}). But (\ref{baricentic}) implies
\begin{align*}
\phi_{FG}(0,x_1+1,\ldots,x_n+1)=\phi_{FG}(0,x_1,\ldots,x_n)+\lambda_0>0.
\end{align*}

\medskip\noindent(c) This part is equivalent to the claim that the automorphism
\begin{align*}
\phi:\RR^{d+1}&\to\RR^{d+1},\\
(v_i,x_i)&\mapsto(v_i,1),\qquad i=1,\ldots,d,\\
\ee_{d+1}&\mapsto\ee_{d+1},
\end{align*}
maps $C(P,\bar x)$ to a cone $C(P,\bar y)$ for some $\bar y\in\SF^0(P)$. But this follows from the fact that the polytopal ball $\phi\big(\B(C(P,\bar x))\big)$ is convex towards $0$.

\medskip\noindent(d) Let $\bar x=(0,\ldots,0,x_{d+1},\ldots,x_n)\in\SF^+(P)$ and $\rho(\bar x)=\bar y=(0,\ldots,0,y_{d+1},\ldots,y_n)$. Pick $k\in\{d+1,\ldots,n\}$. We have $v_k=\sum_{i=1}^d a_iv_i$ for some uniquely determined $a_1,\ldots,a_d\in\RR$. The linear dependence in $\RR^{d+1}$
\begin{align*}
\sum_{i=1}^d a_i(v_i,x_i+1)=(v_k,x_k+1)+\left(\sum_{i=1}^da_i(x_i+1)-x_k-1\right)\ee_{d+1},
\end{align*}
upon application of $\phi$ from the part (c), implies
\begin{align*}
\sum_{i=1}^d a_i(v_i,1)=(v_k,y_k+1)+\left(\sum_{i=1}^da_i(x_i+1)-x_k-1\right)\ee_{d+1}.
\end{align*}
In particular, we have a linear functional
\begin{equation}\label{linear_functional}
y_k=x_k-\sum_{i=1}^da_ix_i.
\end{equation}

\medskip\noindent(e)  That $\Aut(P)$ acts on the set $\SF^0(P)$ follows from the identities $\rho^2=\rho$ and
\begin{align*}
\rho(\alpha(\beta(\bar x)))=\rho(\alpha(\rho(\beta(\bar x)))),\quad\alpha,\beta\in\Aut(P),\quad\bar x\in\SF(P)^0,
\end{align*}
the latter being a consequence of the uniqueness of $\rho(\bar x)$ in the part (c). But then, by the part (d), $\Aut(P)$ acts on $\SF^0(P)$ by linear automorphisms.
\end{proof}

In the notation, used in Lemma \ref{combined}, we have

\begin{lemma}\label{fundamental}
Assume $P$ is a smooth Fano $d$-polytope with $n$ vertices.
\begin{enumerate}[\rm(a)] 
\item The action of $\Aut(P)$ on $\SF^0(P)$ restricts to an action on $\SF^0(P)\cap\ZZ^{n+1}$.
\item There exists an $(n-d)$-cone $C\subset(\underbrace{0,\ldots,0}_{d+1},\RR^{n-d})$, such that:
\begin{enumerate}[\rm(i)]
\item
$\inte(C)\subset\SF^0(P)$,
\item $C\cap\SF^0(P)\cap\ZZ^{n+1}$ surjects onto $\big(\SF^0(P)\cap\ZZ^{n+1}\big)/\Aut(P)$,
\item
$\inte(C)\cap\SF^0(P)\cap\ZZ^{n+1}$ injects into $\big(\SF^0(P)\cap\ZZ^{n+1}\big)/\Aut(P)$.
\end{enumerate}
\end{enumerate} 
\end{lemma}

\begin{proof}
(a) Since $P$ is a smooth Fano polytope, the action of $\Aut(P)$ on $\RR^d$ respects the integer lattice $\ZZ^d$. Consequently, the $\Aut(P)$-action on $P\times\RR_{\ge0}$ respects the lattice structure and we have the induced action of $\Aut(P)$ on $\SF^+(P)\cap\ZZ^{n+1}$. Using again that $P$ is a smooth Fano polytope, the linear functional (\ref{linear_functional}) is defined over $\ZZ$. In other words, the homomorphism $\rho$ in Lemma \ref{combined}(d) restricts to a homomorhism $\SF^+(P)\cap\ZZ^{n+1}\to\SF^0(P)\cap\ZZ^{n+1}$ and, thus, the action of $\Aut(P)$ on $\SF^0(P)$ respects the lattice structure. 

\medskip\noindent (b) This is a direct consequence of the general construction of a \emph{fundamental domain/cone} for a discrete automorphism group of a cone \cite[Application 4.14]{Fundamental}, applicable to our situation in view Lemma \ref{combined}(e) and the part (a).
\end{proof}

\medskip\noindent\emph{Notice.} Lemma \ref{fundamental} implies that, for a smooth Fano polytope $P$, almost all of the orbit set $\big(\SF^0(P)\cap\ZZ^{n+1}\big)/\Aut(P)$, except a `measure 0 part',  can be made in a systematic way into a full rank sub-semigroup of $\ZZ^{n-d}$, although this structure depends on the choice of the fundamental domain/cone. We expect that the full orbit set also carries a similar semigroup structure, which would follow from the existence of a convex conical \emph{strict} fundamental domain of the mentioned $\Aut(P)$-action. 

\medskip Let $\Sigma$ be a simplicial sphere and $\Fano(\Sigma)$ be the set of unimodular equivalence classes of the smooth Fano polytopes with the boundary complex, equivalent to $\Sigma$. As $\Sigma$ varies over simplicial spheres: always $\#\Fano(\Sigma)<\infty$, often $\#\Fano(\Sigma)\le1$, and in the absolute majority of cases $\#\Fano(\Sigma)=0$ .

Lemmas \ref{combined} and \ref{fundamental} explain the structures, used in the statement of the following:

\begin{theorem}\label{Fano_bottom}
For a simplicial sphere $\Sigma$, as $P$ varies over $\Fano(\Sigma)$ and $\bar x$ varies over $\SF^0(P)\cap\ZZ^{n+1}$, the assignment $\bar x\mapsto C(P,\bar x)$  gives rise to a bijective correspondence between the isomorphism classes of the regular reduced conic realizations of $\Delta(\Sigma)$ and the disjoint union\ \ $\biguplus_{P\in\Fano(\Sigma)} \big(\SF^0(P)\cap\ZZ^{n+1}\big)/\Aut(P)$.
\end{theorem}

\begin{proof}
First we observe that, if $\Delta(\Sigma)$ admits a regular reduced conic realization $C\subset\RR^{d+1}$, then $\Sigma$ is equivalent to the boundary complex of a smooth Fano polytope. Let $O$ be the interior vertex of $\B(C)$ and $H\subset\RR^{d+1}\setminus\{O\}$ an arbitrary hyperplane. Then the image $P$ of $\B(C)$ in $H$ under the parallel projection along the line $\RR O\subset\RR^{d+1}$ is a convex $\Lambda$-polytope, where $\Lambda\subset H$ is the image of $\ZZ^{d+1}$ under the projection -- it is a lattice in $H$. Thinking  of $\Lambda$ as $\ZZ^d$, the polytope $P$ becomes a smooth Fano polytope, whereas the boundary complex of $P$ is equivalent to $\Sigma$.

Assume $P\in\Fano(\Sigma)$. It follows from Lemma \ref{combined}(b) that, for every $\bar x\in\SF^0(P)\cap\ZZ^{n+1}$, the cone $C(P,\bar x)$ is a regular reduced conic realization of $\Delta(\Sigma)$.

Conversely, we claim that, for a regular reduced conic realization  $C\subset\RR^{d+1}$ of $\Delta(\Sigma)$, there exists $P\in\Fano(\Sigma)$ and $\bar x\in\SF(P)^0\cap\ZZ^{n+1}$, such that $C\cong C(P,\bar x)$. Assume $w_1,\ldots,w_n$ are the vertices of $\B(C)$ in the boundary of $\B(C)$ and $w_0$ is the inner vertex of $\B(C)$. Consider the plane $H=\sum_{i=1}^d(w_i-w_0)\subset\RR^{d+1}$. Let $Q$ be the image of $\B(\bar x)$ in $H$ under the parallel projection $\RR^{d+1}\to H$ along the line $\RR w_0\subset\RR^{d+1}$ and $\Lambda$ be that of $\ZZ^{d+1}$. The linear transformation $\ZZ^{d+1}\to\ZZ^{d+1}$, mapping $\Lambda$ to $(\ZZ^d,0)$ and $w_0$ to $\ee_{d+1}$, induces an isomorphism $C\to C(P,\bar x)$, where $P\in\Fano(\Sigma)$ is the image of $Q$.

It remains to show that, for a polytope $P\in\Fano(\Sigma)$ and points $\bar x,\bar y\in\SF(P)^0\cap\ZZ^{n+1}$, the cones $C(P,\bar x)$ and $C(P,\bar y)$ are lattice isomorphic if and only if $\bar x$ and $\bar y$ are in a same $\Aut(P)$-orbit. The `if' part is immediate from the definition of the $\Aut(P)$-action. Assume $\theta:C(P,\bar x)\to C(P,\bar y)$ is a lattice isomorphism, for which
\begin{align*}
(v_k,1)&\mapsto (v_{i_k},y_{i_k}),\ \text{for some}\ i_k,\ \text{where}\ k=1,\ldots,d,\\
\ee_{d+1}&\mapsto\ee_{d+1},
\end{align*}
where the $v_i$ are as in Lemmas \ref{combined} and \ref{fundamental} and $\bar y=(0,\ldots,0,y_{d+1},\ldots,y_n)$. Since $\theta$ leaves the following prism in $\RR^{n+1}$ invariant
\begin{align*}
\bigcup_{z\in P}(z+\RR v_0)=\bigcup_{z\in\B(C(P,\bar x))}(z+\RR v_0)=\bigcup_{z\in\B(C(P,\bar y))}(z+\RR v_0),
\end{align*}
the assignment $v_k\mapsto v_{i_k}$, $k=1,\ldots,d$, gives rise to an automorphism $\alpha\in\Aut(P)$ and, due to the uniqueness of $\rho(x)$ in Lemma \ref{combined}(c),  we have $\alpha\star\bar x=\bar y$.
\end{proof}

\begin{example}\label{explicit_computation}
(a) If $\Sigma$ defines a minimal triangulation of the $(d-1)$-sphere then $\#\Fano(\Sigma)=1$, the corresponding toric variety is $\PP^d_\kk$, and Lemma \ref{fundamental}(b) yields a unique semigroup structure on the full set of isomorphism classes of regular reduced conic realizations of $\Delta(\Sigma)$, which is isomorphic to the additive semigroup $\NN$.

\medskip(b) If $P\subset\RR^d$ is the standard $d$-dimensional cross-polytope, then the corresponding toric variety is $\big(\PP^1_\kk\big)^d$, Lemma \ref{fundamental}(b) yields a unique semigroup structure on the full orbit set $\big(\SF^0(P)\cap\ZZ^{2d+1}\big)/\Aut(P)$, and we have isomorphisms:
\begin{align*}
\big(\SF^0(P)\cap\ZZ^{2d+1}\big)/\Aut(P)\cong&\NN^d/\text{Sym}(d)\cong\\
&\{(a_1,\ldots,a_d)\in\NN^d\ |\ a_1\le\cdots\le a_d\}\cong\NN\times(\ZZ_+)^{d-1}.
\end{align*}
\end{example}

\begin{example}\label{useful}
Let $P\subset\RR^d$ be a smooth Fano polytope. Then
\begin{align*}
(0,\underbrace{k,\ldots,k}_n)\in\SF^+(P)\cap\ZZ^{n+1},\qquad k\in\NN,
\end{align*}
are fixed points of the $\Aut(P)$-action. Moreover, the homomorphism $\rho$ is injective on this set of points. In fact, for the corresponding cones
\begin{align*}
C\big((P,(0,k,k,\ldots,k)\big)\subset\RR^{d+1},\qquad k\in\NN,
\end{align*}
a direct inspection of the relations between the Hilbert basis elements shows that different values of $k$ yield non-lattice-isomorphic cones.
\end{example}

\begin{example}\label{Non_regular_Fano}
If $\Sigma$ is combinatorially equivalent to the boundary complex of a smooth Fano polytope,  $\Delta(\Sigma)$ may well have a \emph{non-regular} bottom reduced conic realization. Consider the cones
\begin{align*}
&C_1=\RR_+(-1,0,1)+\RR_+(0,0,1)+\RR_+(1,0,2)+\RR_+(3,1,2),\\
&C_2=\RR_+(-1,0,1)+\RR_+(0,0,1)+\RR_+(1,0,2)+\RR_+(3,-1,2).
\end{align*}
The maximal facets in the bottom complexes of these cones are, respectively, the following pairs of unimodular triangles, whose affine hulls are on lattice distance one from the origin:
\begin{align*}
&\conv((-1,0,1),(0,0,1),(3,1,2))\quad\&\quad\conv((0,0,1),(1,0,2),(3,1,2)),\\
&\conv((-1,0,1),(0,0,1),(3,-1,2))\quad\&\quad\conv((0,0,1),(1,0,2),(3,-1,2)).
\end{align*}
The two cones can be glued in the sense of Lemma \ref{Patching}, using the element $\gamma=(0,0,1)$ as in the proof of that lemma. By Lemma \ref{alternative}(b), for $t\gg1$, the resulting cone $C_t$ will be a reduced bottom realization of $\Delta(\Sigma)$, where $\Sigma$ is equivalent to the boundary complex of the standard 2-dimensional cross-polytope. Yet, for every $t$, the orthogonal projection of $\B(C_t)$ onto $\RR^2$ is the same \emph{non-convex} set 
$$
\conv((-1,0),(1,0),(3,1))\cup\conv((-1,0),(1,0),(3,-1))\subset\RR^2.
$$
\end{example}

\begin{remark}\label{Fano_construct}
If a simplicial sphere $\Sigma$ is not equivalent to  the boundary complex of a smooth Fano polytope then, according to Theorem \ref{Fano_bottom}, $\Delta(\Sigma)$ admits no regular reduced conic realization. But it may well admit non-regular reduced conic realization. In fact, in Theorem \ref{Delta} below we show that, for \emph{every} simplicial circle $\Sigma$, the complex $\Delta(\Sigma)$ is bottom reduced.  
\end{remark}

\section{Bottom simplicial discs}\label{Bottom_discs}

In Section \ref{Non_stacked} we derive infinitely many reduced bottom simplicial discs, not covered by Theorems \ref{MAIN} and \ref{Fano_bottom}. In view of Lemma \ref{alternative}(b) the following question is reminiscent of the open question whether all complete 3-dimensional simplicial fans are combinatorially equivalent to smooth fans (e.g., \cite{Simplicial_2}):

\begin{question}\label{open}
Is every simplicial disc reduced bottom?
\end{question}

\subsection{Simplicial balls as reduced bottom complexes of affine monoids.} Here we show that, if the normality condition for the underlying monoid is relaxed, then Question \ref{open} has the positive answer.

Recall, a simplicial complex $\Delta$ is called regular if it is combinatorially equivalent to a regular triangulation of a polytope (Section \ref{Obstructions}).

Let $\kk$ be a field. For an affine positive monoid $M$, denote by $\gr(\kk[M])$ the associated graded algebra with respect to the maximal monomial ideal in $\kk[M]$. 

\begin{theorem}\label{evidence}
\begin{enumerate}[\rm(a)]
\item A simplicial ball $\Delta$ is regular if and only if there exists an affine positive monoid $M$, such that $\kk[\Delta]\cong\gr(\kk[M])$ as graded $\kk$-algebras.
\item Every two-dimensional simplicial disc is regular. 
\end{enumerate}
\end{theorem}

In the proof we will use several times the following projective map
\begin{align*}
\pi:\PP_\RR^{d+1}\to\PP_\RR^{d+1},\qquad [x_0:x_1:\cdots:x_d:x_{d+1}]\mapsto[x_{d+1}:x_1:\cdots:x_d:x_0],
\end{align*}
where $\RR^{d+1}\equiv(X_0=1)$. 

\begin{proof}(a) Assume $\kk[\Delta]\cong\gr(\kk[M])$ as graded $\kk$-algebras for an affine positive monoid $M\subset\ZZ^d$. Then the same approach as in the proof of Theorem \ref{Folklore} shows that $\Delta$ is combinatorially equivalent to the geometric simplicial complex, defined by the compact faces of $\conv(M\setminus\{0\})\subset\RR^d$, and the same argument as in Section \ref{Obstructions} for the special case $M=C\cap\ZZ^d$, where $C$ is a cone, implies that $\Delta$ is regular. 

Conversely, let  $\Delta$ be combinatorially equivalent to a regular triangulation $\mathcal T$ of a $d$-polytope $P\subset\RR^d$. Let $f:P\to\RR$ be a support function for $\mathcal T$ and $\Gamma_f\subset\RR^{d+1}$ be its graph. Without loss of generality we can assume $f(P)\subset\RR_{>0}$ and $0\in\inte(P)$.

First we show that without loss of generality $P$ can be assumed to be a simplicial polytope.  For a sufficiently small number $t<0$, the polytope
\begin{align*}
Q=\bigg(\bigcup_{x\in\Gamma_f}\big[x,t\ee_{d+1}\big]\bigg)\ \bigcap\ (\RR^d,0)
\end{align*}
is simplicial and the projection of $\mathcal T$ into $(\RR^d,0)$ from the pole $t\ee_{d+1}$ induces a triangulation of $Q$, combinatorially equivalent to $\Delta$. The map $\pi$ transforms the subset
\begin{align*}
\bigg(\bigcup_{x\in\Gamma_f}\big[x,t\ee_{d+1}\big]\bigg)-t\ee_{d+1}\subset\RR^d\times\RR_+
\end{align*}
into a set of the form $\bigcup_{x\in\Gamma\rq{}}\big(x-\RR_+\ee_{d+1}\big)$, where $\Gamma\rq{}\subset\RR^d\times\RR_{>0}$ is the corresponding transform of $\Gamma_f$. But then the polytope
\begin{align*}
R=\bigg(\bigcup_{x\in\Gamma\rq{}}\big(x-\RR_+\ee_{d+1}\big)\bigg)\ \bigcap\ (\RR^d,0)
\end{align*}  
is simplicial and the piece-wise affine function $R\to\RR$, whose graph is $\Gamma\rq{}$, supports a triangulation of $R$, combinatorially equivalent to $\Delta$.  

Having achieved that $\Delta$ is combinatorially equivalent to a regular triangulation $\mathcal T$ of a simplicial $d$-polytope $P\subset\RR^d$, supported by a function $f:P\to\RR_{>0}$, we can further assume that $P$ and the simplices in $\mathcal T$ are rational and so are the values of $f$ at the vertices of $\mathcal T$. The map $\pi$ transforms the infinite prism $P\times\RR$ into a cone $C\subset\RR^d\times\RR_{>0}$. The image $\pi(\Gamma_f)$ of the graph $\Gamma_f$ of $f$ is a polytopal ball, convex towards $0$ and with rational vertices. We have $\RR_+\pi(\Gamma_f)=C$. Let $k\pi(\Gamma_f)$ be the dilation of $\pi(\Gamma_f)$ by a factor $k\in\NN$, making the vertices of $\pi(\Gamma_f)$ into lattice points. Consider the affine monoid $M\subset\ZZ^{d+1}$, generated by the vertices of the polytopal ball $k\pi(\Gamma_f)$. The same approach as in the proof of Theorem \ref{Folklore} shows $\kk[\Delta]\cong\gr(\kk[M])$.

\medskip\noindent(b) Assume $\Delta$ is a simplicial disc. The classical \emph{Steinitz Theorem} on the 1-skeleton of a 3-polytope \cite[Section 15]{Stacked} implies that every simplicial 2-sphere is polytopal; see also the discussion in \cite[p. 503]{Triang}. Let $\Sigma(\Delta)$ be a simplical sphere, obtained by taking the pyramids with a common vertex over the boundary faces of $\Delta$, and $P\subset\RR^3$ be a simplicial $3$-polytope, whose boundary complex is combinatorially equivalent to $\Sigma(\Delta)$. Assume $v\in\vertex(P)$ represents the vertex in $\Sigma(\Delta)\setminus\Delta$. We can assume $v=0$ and $P\in\RR^2\times\RR_{>0}$. The map $\pi$ transforms the cone $\RR_+P\subset\RR^3$ into an infinite prism. Moreover, the union of faces of $P$, not incident with $0$, is transformed into the graph of a piece-wise affine function $h$. Then the function $-h$ supports a triangulation of the orthogonal cross-section of this prism by $(\RR^2,0)$. But this triangulation is combinatorially equivalent to $\Delta$.
\end{proof}

\subsection{A criterion for bottom simplicial discs}\label{2DIM}  
Here we derive a criterion for a simplicial disc to be reduced bottom. It leads to an algorithm which, if implemented, can be used for analyzing many explicit examples. The idea for the criterion is based on the approach in \cite{Klein}.

Let $\Delta$ be a reduced bottom simplicial disc on the vertex set $\{1,\ldots,n\}$. Assume $\iota:\Delta\to\RR^3$ is an embedding, satisfying the condition in Lemma \ref{alternative}(b), and $C\subset\RR^3$ is the corresponding cone. Put $[i]=\iota(i)$, where $i=1,\ldots,n$. Then the points $[i]$ are subject to relations with integer coefficients of the form:
\begin{equation}\label{Relations}
[i]+[j]+\lambda_{kl}[k]+\mu_{kl}[l]=0,\qquad i<j,\ k<l,
\end{equation}
where:
\begin{enumerate}[\rm(i)]
\item $\{i,j,k\},\ \{j,k,l\}\in\Delta$: this is the condition that $\{[i],[k],[l]\}$ is a basis for $\ZZ^3$ if and only if $\{[j],[k],[l])\}$ is;
\item $\lambda_{kl}+\mu_{kl}\le-3$: this is a part of the convexity condition towards $0$ for $|\iota(\Delta)|$;
\item $\lambda_{kl}>0$ whenever $[l]$ is on the boundary of $|\iota(\Delta)|$ and symmetrically for $\mu_{kl}$: this is the other part of the convexity condition towards $0$, namely that the conical set $\RR_+|\iota(\Delta)|$ is a cone.
\end{enumerate}

\noindent\emph{Notice.} When the boundary of $\Delta$ consists of only three segments the condition (iii) is redundant.

\medskip Conversely, if a system of nonzero vectors $v_1,\ldots,v_n\in\RR^3$ satisfies the relations above upon substituting $v_i\mapsto[i]$, then the assignment $i\mapsto v_i$ gives rise to an embedding $\iota:\Delta\to\RR^3$, satisfying the conditions in Lemma \ref{alternative}(b), where the lattice of reference is changed from $\ZZ^3$ to $\ZZ v_1+\cdots+\ZZ v_n\cong\ZZ^3$.

Consider the following $m\times n$-matrix $\MM_\Delta$, where $m$ is the number of edges in $\Delta$ that are shared by two triangles in $\Delta$: the $\alpha\beta$-entry is the coefficient of $[\beta]$ in the corresponding linear relation (\ref{Relations}). We consider these relations as linear relations involving all points $[1],\ldots,[n]$, some with 0 coefficients. Let $L_\Delta:(\RR^3)^n\to(\RR^3)^m$ be the linear map
\begin{align*}
(w_1,\ldots,w_n)\mapsto\left(\sum_{\beta=1}^n m_{\alpha\beta}w_\beta\right)_{\alpha=1}^m
\end{align*}
Then $\ker(L_\Delta)$ consists the $n$-tuples $(v_1,\ldots,v_n)$ of vectors in $\RR^3$, satisfying the relations (\ref{Relations})(i,ii,iii) upon substituting $v_i\mapsto[i]$. Without loss of generality we can assume that $[1]$, $[2]$ and $[3]$ are linearly independent. Then such an $n$-tuple $(v_1,\ldots,v_n)$ is uniquely determined by the triple $(v_1,v_2,v_3)$, and the latter can be arbitrary. Consequently, $\dim\ker(L_\Delta)=9$ or, equivalently, $\rank(\MM_\Delta)=\frac{3n-9}3=n-3$.

We have proved the following lemma, which will be used in the next section:

\begin{lemma}\label{criterion}
Let $\Delta$ be a simplicial disc on the vertex set $\{1,\ldots,n\}$. The isomorphism classes of reduced conic realizations of $\Delta$ is bijective to the set of integer $m\times n$-matrices ${\mathbb M}_\Delta$, where $m$ is the number of edges $\{k,l\}\in\Delta$ with $k<l$, admitting vertices $i,j\in\Delta$ with $\{i,k,l\},\ \{j,k,l\}\in\Delta$, and the entries are
\begin{align*}
&m_{\alpha\beta}=
\begin{cases}
&1\ \text{if}\ \beta=\{i,j\},\\
&\lambda_{kl}\ \text{if}\ \beta=k,\\
&\mu_{kl}\ \text{if}\ \beta=l,\\
&0\ \text{if}\ \beta\notin\{i,j,k,l\},\\
\end{cases},\\
&\qquad\qquad\qquad\qquad\qquad\qquad\qquad
\text{where}\ \ \alpha=\{k,l\},
\end{align*}
subject to the relations:
\begin{enumerate}[\rm(a)]
\item $\lambda_{kl}+\mu_{kl}\le-3$,
\item $\lambda_{kl}>0$ whenever $l$ is on the boundary of $\Delta$ and symmetrically for $\mu_{kl}$,
\item $\rank({\mathbb M}_\Delta)=n-3$.
\end{enumerate}
\end{lemma}

\subsection{Non-stacked bottom simplicial discs}\label{Non_stacked} Let $\Delta$ be the simplicial disc with 6 vertices as shown on Figure 4 on the left and $\Delta_n$ be the simplical complex of the stellar triangulation of the regular $n$-gon:

\begin{figure}[h!]
\caption{Complexes $\Delta$ and $\Delta_6$}
\vspace{.25in}
\includegraphics[scale=5]{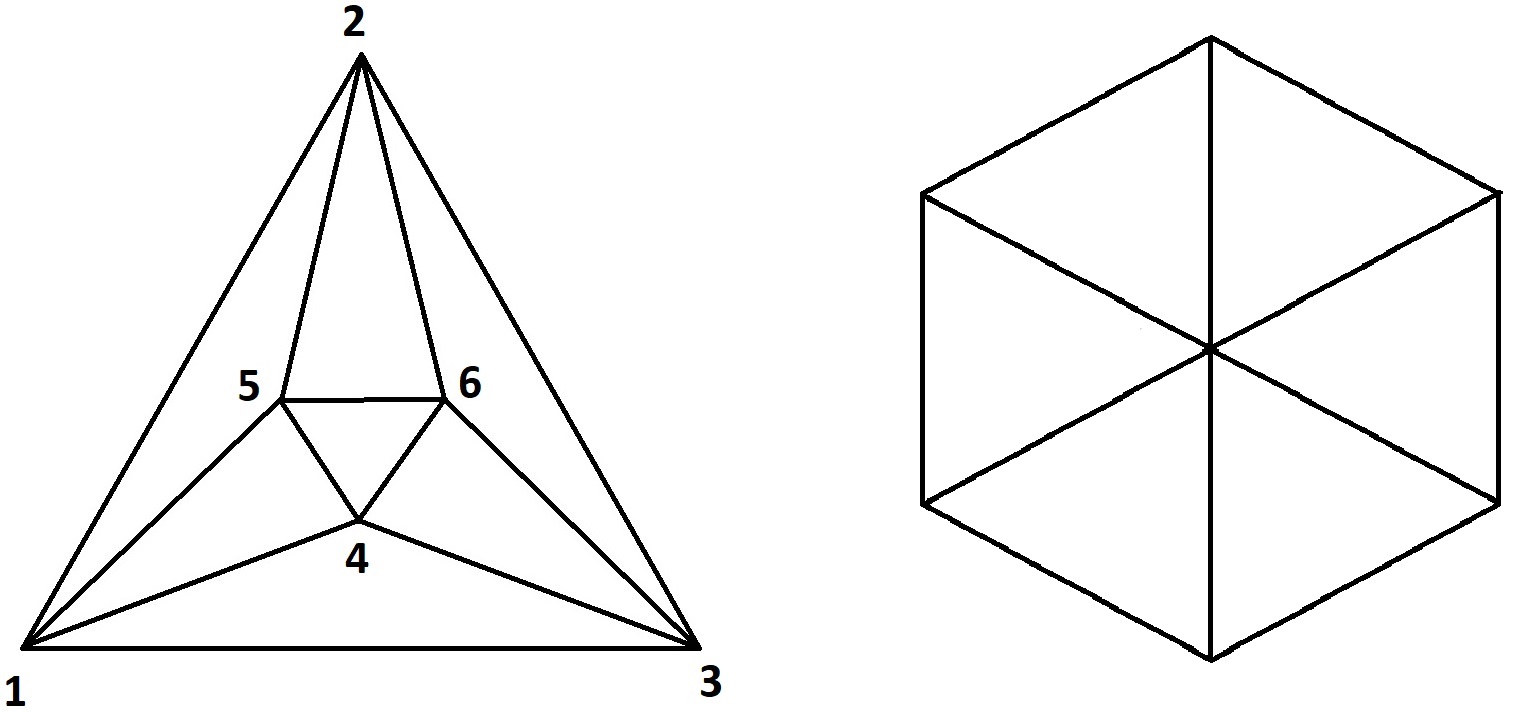}
\end{figure}

\begin{theorem}\label{Delta}
\begin{enumerate}[\rm(a)]
\item The isomorphism classes of reduced conic realizations of $\Delta$ are in a natural bijective correspondence with $(\ZZ_{\ge3})^3$.
\item $\Delta_n$ is reduced bottom for every $n\ge3$.
\end{enumerate}
\end{theorem}

\begin{proof} (a) Using the notation, introduced in Lemma \ref{criterion}, the matrix $\MM_\Delta$ is

\medskip
\begin{center}
\begin{tabular}{|c || c | c | c | c | c | c | }
\hline
${\mathbb M}_\Delta$ & 1 & 2 & 3 & 4 & 5 & 6\\
\hline
\hline
14 & $\lambda_{14}$ & 0 & 1 & $\mu_{14}$ & 1 & 0\\ 
\hline
15 & $\lambda_{15}$ & 1 & 0 & 1 & $\mu_{15}$ & 0\\  
\hline
25 & 1 & $\lambda_{25}$ & 0 & 0 & $\mu_{25}$ & 1\\
\hline
26 & 0 & $\lambda_{26 }$ & 1 & 0 & 1 & $\mu_{26}$\\ 
\hline
36 & 0 & 1 & $\lambda_{36}$ & 1 & 0 & $\mu_{36}$\\  
\hline
34 & 1 & 0 & $\lambda_{34}$ & $\mu_{34}$ & 0 & 1\\
\hline
46 & 0 & 0 & 1 & $\lambda_{46}$ & 1 & $\mu_{46}$\\ 
\hline
54 & 1 & 0 & 0 & $\lambda_{45}$ & $\mu_{45}$ & 1\\  
\hline
56 & 0 & 1 & 0 & 1 & $\lambda_{56}$ & $\mu_{56}$ \\
\hline
\end{tabular}
\end{center}

Lemma \ref{criterion}(c) implies that every row in this matrix must be an integral linear combination of the last three rows. Lemma \ref{criterion}(a,b) leads to a number of constrains on the entries of ${\mathbb M}_\Delta$. They are amenable to an effective analysis (by hand), showing that  every row in ${\mathbb M}_\Delta$ must have exactly three non-zero entries. Consequently, for any reduced conical realization $C$ of $\Delta$ and the \emph{geometric} simplicial complex $\Delta_C$, obtained by intersecting the cones, determined by the triangles in $\Delta$, with an affine plane, meeting $C$ transversally, there are only two possibilities, shown in Figure 5.
\begin{figure}[h!]\label{Delta'}
\caption{Geometric simplicial complex $C_\Delta$}
\vspace{.25in}
\includegraphics[scale=5.5]{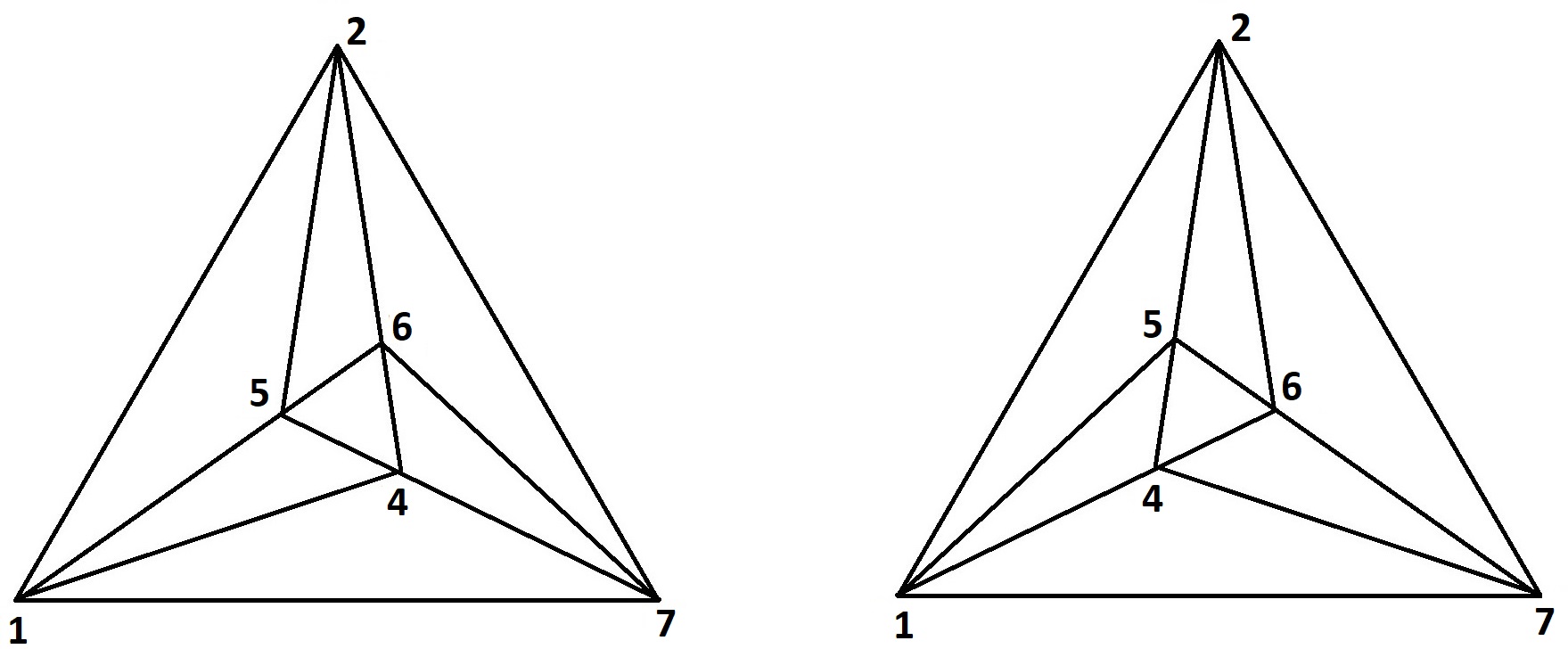}
\end{figure}
These two possibilities lead to isomorphic reduce conic realizations. Finally, the isomorphism classes of the reduced conic realizations, corresponding to the left geometric simplicial complex in Figure 5, are determined by the matrices
$$
{\mathbb M}_\Delta=\begin{pmatrix}
0 & 0 & 1 & -z & 1 & 0\\ 
y & 1 & 0 & 1 & yx & 0\\  
1 & 0 & 0 & 0 & x & 1\\
0 & z & 1 & 0 & 1 & -yz\\ 
0 & 1 & 0 & 1 & 0 & -y\\  
1 & 0 & -x & xz & 0 & 1\\
0 & 0 & 1 & -z & 1 & 0\\ 
1 & 0 & 0 & 0 & x & 1\\  
0 & 1 & 0 & 1 & 0 & -y\\
\end{pmatrix}
$$
where $x,y,z$ are arbitrary integers $\ge3$. This proves Theorem \ref{Delta}(a).

\medskip An example of a reduced conic realization of $\Delta$ of the type on the left in Figure 5 is given by

\begin{center}
\begin{tabular}{l l l l}
$[1]=(1,0,0)$, & & & $[4]=(0,1,0)$,\\
$[2]=(-3,-1,9)$, & & & $[5]=(0,0,1)$,\\ 
$[3]=(0,3,-1)$, & & & $[6]=(-1,0,3)$.\\
\end{tabular}
\end{center}

\medskip\noindent(b) For $\Delta_3$, $\Delta_4$, and $\Delta_5$, we have the reduced conic realizations
\begin{align*}
C(F_i):=\sum_{\vertex(F_i)}\RR_+(v,2),\quad i=3,4,5,
\end{align*}
where:
\begin{align*}
&F_3=\conv(\ee_1,\ee_2,-\ee_1-\ee_2),\\
&F_4=\conv\big(\ee_1,\ee_2,-\ee_1,-\ee_2),\\ &F_5=\conv\big(\ee_1,\ee_2,-\ee_1,-\ee_2,\ee_1+\ee_2).
\end{align*}

Figure 6 represents $45^\text{o}$-rotations of $F_4$ and $F_5$.
 
\medskip\noindent\emph{Notice.} The cones $C(F_i)$ are the special cases of Example \ref{useful}, corresponding to the smooth Fano polygons $F_3,F_4,F_5$, and the parameter $k=1$.

\begin{figure}[h!]
\caption{Smooth Fano polygons $F_4$ and $F_5$.}
\vspace{.15in}
\includegraphics[scale=7]{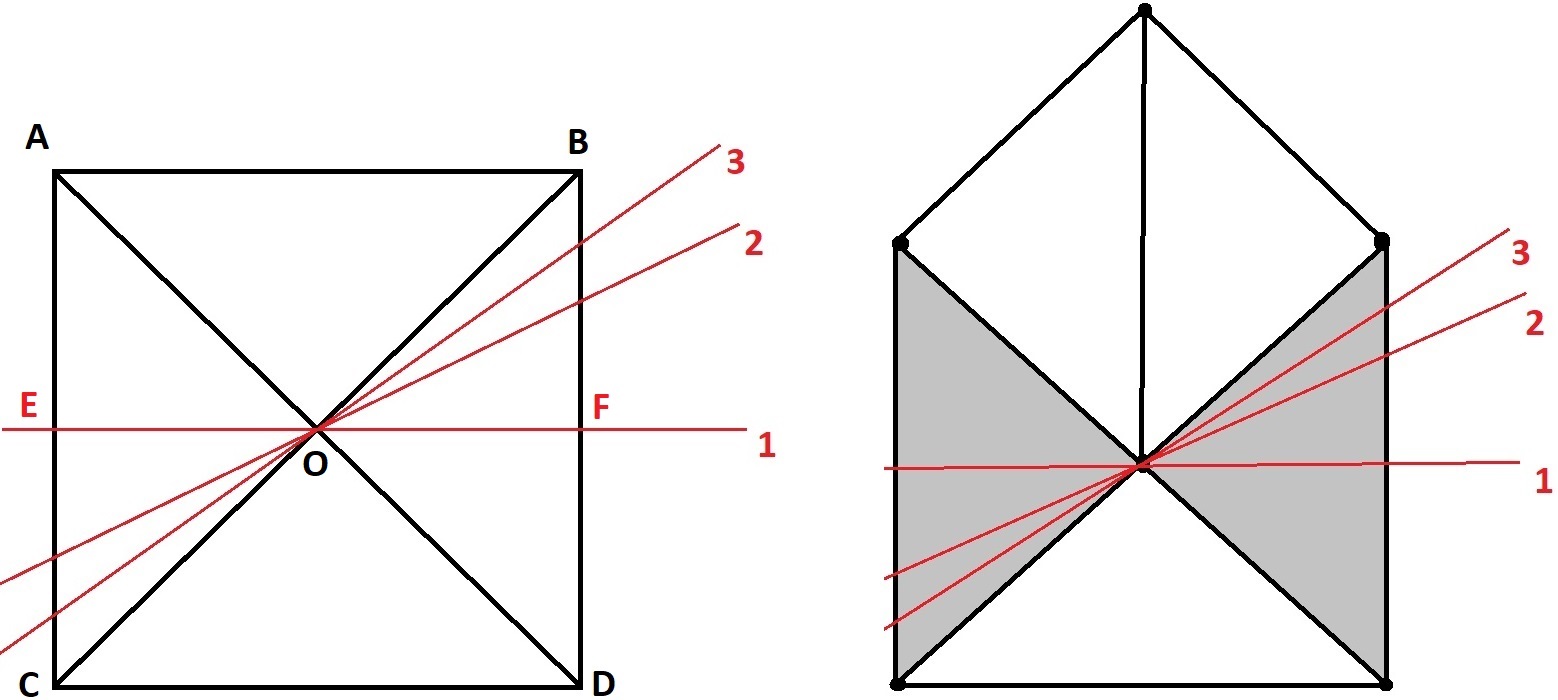}
\end{figure}

\medskip\noindent\emph{Case:} $n$ even. The bottom $\B(F_4)$ consists of the triangles 
$$
\conv(A,O,C),\ \conv(A,O,B),\ \conv(B,O,D),\ \conv(C,O,D),
$$
where $O=\ee_3$. Consider the plane $H\subset\RR^3$ through $0\in\RR^3$ and $O$, that cuts the triangles $\conv(AOC)$ and $\conv(BOD)$ exactly in half, as shown on Figure 6; the plane $H$ is represented by the line labeled 1. Let $C^+$ be the half of the cone $C(F_4)$ containing $A,B$ and $C^-$ that, containing $C,D$. Then the primitive lattice point in the direction of $E=[AC]\cap H$ is $A+C$ and that in the direction of $F=[BD]\cap H$ is $B+D$. One has:
\begin{align*}
&\B(C^+)=\conv(A,O,A+C)\cup\conv(A,O,B)\cup\conv(B,O,B+D),\\
&\B(C^-)=\conv(C,O,C+C)\cup\conv(C,O,D)\cup\conv(D,O,B+D).
\end{align*}
In particular, $\BB(C^+)$ and $\BB(C^-)$ can be conically glued along their common part in the sense of Section \ref{REALIZATION}. This gluing can be carried out so that the $180^0$-rotational symmetry about the axis $\RR O\subset\RR^3$  between $\conv(C,O,E)$ and $\conv(B,O,F)$ is respected: one chooses the bases $\mathcal B_1$ and $\mathcal B_2$ in the proof of Lemma \ref{Patching} in the form:
\begin{align*}
\mathcal B_1=\{u,O,w\},\qquad\mathcal B_2=\{-u,O,w\},
\end{align*}
where $u\in(\RR^2,0)\setminus H$, $w\in(\RR^2,0)\cap H$, and $\gamma=O$. The result of this gluing is a reduced conic realization of the simplicial complex $\BB(C^+)\vee\BB(C^-)$, which contains a pair of unimodular triangles, invariant under the $180^{\text{o}}$ rotation about $\RR O$: they correspond to $\conv(C,O,E)$ and $\conv(B,O,F)$. The complex $\BB(C^+)\vee\BB(C^-)$ is combinatorially equivalent to the triangulation of the square $\conv(A,B,C,D)$ into 6 triangles, sharing the vertex $O$ (Figure 6). Next one carries out the similar process of `cracking in half' with respect to this pair of triangles, leading to a reduced bottom realization of the simplicial disc with 8 facets, combinatorially equivalent to the dissection of the previous triangulation of $\conv(A,B,C,D)$ in Figure 6 by the line labeled 2. By iterating the process, one proves Theorem \ref{Delta} for $n\ge4$ even.

\medskip\noindent\emph{Case:} $n$ odd. The same argument we used for $n$ even applies to the shaded pair of triangles in $F_5$ (Figure 6), yielding the result for $n\ge5$ odd.
\end{proof}

\medskip\noindent\emph{Acknowledgment.} Thanks to Tamara Mchedlidze for a helpful discussion on convex graph drawing that motivated Theorem \ref{evidence}.

\bibliographystyle{plain}

\bibliography{bibliography}

\end{document}